\documentclass[amsthm]{amsart}
\usepackage{amsmath,amssymb,amsfonts,latexsym,epsfig}
\usepackage{bbm,vector}
\usepackage[naturalnames]{hyperref}
\usepackage{enumerate}
\usepackage{mathrsfs}
\usepackage{color}
\usepackage[english, algosection, algoruled, noline]{algorithm2e}

\newtheorem{theorem}{Theorem}[section]

\newtheorem{definition}[theorem]{Definition}
\newtheorem{proposition}[theorem]{Proposition}

\newtheorem{conjecture}[theorem]{Conjecture}

\newenvironment{keyword}{\par\noindent{}\textbf{Keywords:}}{}

\def\un{\mathbf{1}}
\def\kk{\mathbb{K}}
\def\CC{\mathbb{C}}
\def\NN{\mathbb{N}}
\def\RR{\mathbb{R}}

\def\db{\mathbf{d}}
\def\hb{\mathbf{h}}
\def\lb{\mathbf{l}}

\def\xb{\mathbf{x}}
\def\xa{\underline{\mathbf{x}}}

\def\Proj{\mathbb{P}}
\def\PP{\Proj}
\def\Ac{\mathcal{A}}
\def\Tc{\mathcal{T}}
\def\AA{\mathcal{A}_{\Lambda}}
\def\<{\langle}
\def\>{\rangle}

\newcommand{\dual}[1]{{#1}^{*}}
\def\dAA{\dual{\mathcal{A}}_{\Lambda}}
\def\dRR{\dual{ R}}
\def\SV{\Xi}

\def\Hom{\operatorname{Hom}}
\def\rank{\operatorname{rank}}

\def\mat#1{\mathbb#1}
\def\form#1{\mathcal#1}

\def\set#1{\{ #1 \}}
\def\Set#1{\left\{ #1 \right\}}

\title{General Tensor Decomposition, Moment Matrices and Applications}

\author{A. Bernardi$^{*}$ \& J. Brachat$^{*}$ \& P. Comon$^{+}$ \& B. Mourrain$^{*}$}
\address{
$^{*}$ Alessandra Bernardi \& J\'er\^ome Brachat \& Bernard Mourrain:
GALAAD, INRIA M\'editerran\'ee \\
Sophia-Antipolis, France\\
}
\email{[FirstName.LastName]@inria.fr}

\address{$^{+}$ Pierre Comon:\\Laboratoire I3S, CNRS and Univ. of Nice,\\ 
Sophia-Antipolis, France\\
}
\email{pcomon@i3s.unice.fr}

\begin{document}

\maketitle

\begin{abstract}
The tensor decomposition addressed in this paper may
be seen as a generalisation of Singular Value Decomposition of
matrices. We consider general multilinear and multihomogeneous tensors.
We show how to reduce the problem to a truncated moment matrix problem and
give a new criterion for flat extension of Quasi-Hankel matrices. We connect
this criterion to the commutation characterisation of border bases.
A new algorithm is described. It applies for general multihomogeneous
tensors, extending the approach of J.J. Sylvester to 
binary forms. An example illustrates the algebraic operations involved 
in this approach and how the decomposition can be recovered from
eigenvector computation. 
\end{abstract}

 

\section{Introduction}

Tensors are objects that appear in various contexts and applications.  
Matrices are tensors of order two, and are better known than tensors.
But in many problems, higher order tensors are naturally used to
collect information which depend on more than two variables. Typically, these data could be observations of some experimentation or of a physical phenomenon that depends on several parameters. These observations are stored in a structure called tensor, whose dimensional parameters (or modes) depend on the problem.

The tensor decomposition problem consists of decomposing a tensor (e.g. the set of observations) into a minimal sum of so-called \emph{decomposable} tensors (i.e. tensors of rank $1$). Such a decomposition which is independent of the coordinate
system allows to extract geometric or invariant properties associated with
the observations. For this reason, the tensor decomposition problem has a
large impact in many applications. The first well
known case is encountered for matrices (i.e. tensors of order 2), and  is related to Singular Value Decomposition with
applications e.g. to Principal Component Analysis. 
Its extension to higher order tensors appears in Electrical Engineering
\cite{SwamGS94:ieeesp}, in Signal processing \cite{ComoR06:SP},
\cite{CichA02}, in Antenna Array Processing \cite{DogaM95a:ieeesp}
\cite{ChevAFC05:ieeesp} or Telecommunications \cite{VandP96:ieeesp},
\cite{Chev99:SP}, \cite{SidiGB00:ieeesp}, \cite{FerrC00:ieeesp},
\cite{DelaC07:SP}, in Chemometrics \cite{Bro97:cils} or Psychometrics
\cite{KierK91:psy}, in Data Analysis \cite{Como92:elsevier},
\cite{Card98:procieee}, \cite{Dono01:ieeeit}, \cite{JianS04:ieeesp}, \cite{SmilBG04}, but also
in more theoretical domains such as Arithmetic complexity \cite{Krus77:laa}
\cite{BiniCRL79:ipl} \cite{Stra83:laa} \cite{Land08:bams}. 
Further numerous applications of tensor decompositions may be found in
\cite{CichA02} \cite{ComoJ10} \cite{SmilBG04}. 

From a mathematical point of view, the tensors that we will consider are
elements of $\Tc:=S^{\delta_{1}}(E_{1})\otimes \cdots \otimes
S^{\delta_{k}}(E_{k})$ where $\delta_{i} \in \NN$, $E_{i}$ are vector spaces
of dimension $n_{i}+1$ over a field $\kk$ (which is of characteristic $0$ and
algebraically closed), and $S^{\delta_{i}}(E_{i})$ is the
$\delta_{i}^{\mathrm{th}}$ symmetric power of $E_{i}$.  The set of tensors of
rank $1$ form a projective variety which is called the Veronese variety when $k=1$
or the Segre variety when $\delta_{i}=1, i=1,\ldots,k$.  We will call it
hereafter the Segre-Veronese variety of $\PP(\Tc)$ and denote it $\SV(\Tc)$. 
The set of tensors which are the linear combinations of $r$ elements of the Segre-Veronese
variety are those which admit a decomposition with at most $r$ terms of rank
$1$ (i.e. in  $\SV(\Tc)$). The closure of this set is called the $r$-secant variety and denoted
$\SV_{r}(\Tc)$.  More precise definitions of these varieties will be given in
Sec. \ref{sec:geometric}.

Decomposing a tensor $T$ consists of finding the minimal $r$ such that this
tensor is a sum of $r$ tensors of rank $1$. This minimal $r$ is called the
rank of $T$. By definition, a tensor of rank $r$ is in the secant variety
$\SV_{r}(\Tc)$. Thus analysing the properties of these secant varieties 
and their characterisation helps determining tensor ranks and decompositions.

The case where $k=2$ and $\delta_{1}=\delta_{2}=1$ corresponds to the matrix
case, which is well known. The rank of a matrix seen as a tensor of order
$k=2$ is its usual rank.
The case where $k=1$ and $\delta_{1}=2$ corresponds to the case of quadratic forms
 and is also well understood. The rank of the symmetric tensor is the
usual rank of the associated symmetric matrix.
The case where $k=1$, $\delta_{1}\in\NN$ and $n_{1}=1$ corresponds to binary
forms, which has been analyzed by J.J. Sylvester in \cite{sylv-cr-1886}. A more complete
description in terms of secant varieties is given in \cite{LOarxiv10}.

On our knowledge if $k> 1$ and if at least one of the $\delta_i$'s is larger
than $1$, then there is no specific result in the literature on the defining
ideal of secant varieties of Segre-Veronese varieties 
$\SV(S^{\delta_{1}}(E_{1})\otimes \cdots  \otimes S^{\delta_{k}}(E_{k}))$
except for \cite{CGG} where the authors conjecture that when
$\SV_{r}(S^{\delta_{1}}(E_{1})\otimes  S^{\delta_{2}}(E_{2})))$ is a defective
hypersurface, then its defining equation is a determinantal equation.

In the case of the secant varieties of Veronese varieties (i.e. if $k=1$ and
$\delta_1>1$), the knowledge of their ideal is sparse. Beside the classical results (see one for all \cite{ia-book-1999}) we quote \cite{LOarxiv10first} as the most up-to-date paper on that subject. 
We also quote \cite{LOarxiv10} for a modern approach to equations of secant varieties in general using representation theory.

About the case of secant varieties of Segre varieties (i.e. $\delta_i=1$ for
$i=1, \ldots , k$) the only obvious case is the $2$ factors Segre. For
some of the non trivial  cases  in which equations of secant varieties
of Segre varieties are known we refer to \cite{LandM04},
\cite{LandW07}, \cite{AR08}, \cite{CGG}. 

The first method to compute such a decomposition, besides the case of matrices
or quadratic forms which may go back to the Babylonians, is due to Sylvester
for binary forms \cite{sylv-cr-1886}. Using apolarity, kernels of
catalecticant matrices are computed degree by degree until a polynomial with
simple roots is found. See also \cite{cm-binary-arxiv-2001}, \cite{ia-book-1999}. An
extension of this approach for symmetric tensors has been analyzed in
\cite{ia-book-1999}, and yields a decomposition method in some cases (see
\cite{ia-book-1999}[p. 187]). Some decomposition methods are also available
for specific degrees and dimensions, e.g. using invariant theory
\cite{cm-sp-1996}. In \cite{BGI11}, there is a simplified version of 
Sylvester's algorithm, which uses the mathematical interpretation of the problem
in terms of secant varieties of rational normal curves. The same approach is
used in \cite{BGI11} to give algorithms for the decompositions of symmetric
tensors belonging to $\SV_2(S^{d}(E))$ and to $\SV_3(S^d(E))$. In \cite{BB10} a
complete rank stratification of $\SV_4(S^d(E))$ is given.

In \cite{BracCMT09:laa}, Sylvester's approach is revisited from an affine
point of view and a general decomposition method based on a flat extension
criterion is described. The main contribution of the current paper is to extend this method to more general
tensor spaces including classical multilinear tensors and multihomogeneous
tensors. In particular we give a new and more flexible criterion for the existence of a
decomposition of a given rank, which extends non trivially the result in \cite{ML08} and the
characterization used in \cite{BracCMT09:laa}. This criterion is a rank condition of an
associated Hankel operator. Moreover we use that criterion to write a new algorithm which checks, degree by degree, if the roots deduced from the kernel of the Hankel operator are simple. This allows to compute the rank of any given partially symmetric tensor. 

This paper is an extended version of \cite{BBCM:2011:issac}, with the
complete proofs and with detailed examples.

In Sec.~2, we recall the notations, the geometric point related to secants of
Segre and Veronese varieties, and the algebraic point of view based on moment matrices.
In Sec. 3, we describe the algorithm and the criterion used to solve the
truncated moment problem. In Sec. 4, an example of tensor decompositions from Antenna Array Processing
illustrates the approach.
\section{Duality, moment matrices and tensor decomposition}
\subsection{Notation and preliminaries} 

Let $\kk$ be an algebraically closed field (e.g. $\kk=\CC$ the field of
complex numbers). We assume that $\kk$ is of characteristic $0$.
For a vector space $E$, its associated projective space is denoted
$\Proj(E)$. For $\bvec{v}\in E-\{0\}$ its class in $\Proj(E)$ is denoted
$\overline{\bvec{v}}$.
Let $\Proj^n$ be the projective space of  $E:=\kk^{n+1}$.
For a subset $F = \set{f_1, \dots, f_m}$ of a vector-space (resp. ring) $R$,
we denote by $\<F\>$ (resp. $(F)$) the vector space (resp. ideal) generated
by $F$ in $R$. 

We consider hereafter the symmetric $\delta$-th power $S^{\delta}(E)$ where
$E$ is a vector space of basis $x_{0},\ldots,x_{n}$. An element of $S^{\delta}(E)$ is a
homogeneous polynomial of degree $\delta \in \NN$ in the variables
$\xb=(x_{0},\ldots,x_{n})$.
For $\xb_{1}=(x_{0,1},\ldots,x_{n_{1},1})$ $,\ldots,$
$\xb_{k}=(x_{0,k},\ldots,x_{n_{k},k})$,
 $S^{\delta_{1}}(E_{1})\otimes \cdots \otimes S^{\delta_{k}}(E_{k})$ (with $E_{i}=\<x_{0,i},\ldots,x_{n_{i},i}\>$)
is the vector space of polynomials multihomogeneous of degree $\delta_{i}$ in the 
variables $\xb_{i}$.

Hereafter, we will consider the deshomogeneisation of elements in
$S^{\delta_{1}}(E_{1})\otimes \cdots \otimes S^{\delta_{k}}(E_{k})$, obtained by
setting $x_{0,i}=1$ for $i=1,\ldots,k$. We denote by $R_{\delta_{1},\ldots,\delta_{k}}$
this space, where $R=\kk[\xa_{1},\ldots, \xa_{k}]$ is the space of
polynomials in the variables $\xa_{1}=(x_{1,1},\ldots,x_{n_{1},1}),\ldots,
\xa_{k}=(x_{1,k},\ldots,x_{n_{k},k})$. 

For $\alpha_{i} = (\alpha_{1,i}, \ldots, \alpha_{n_{i},i}) \in \NN^{n_{i}}$ ($i=1,\ldots,k$),  
let $\xa_{i}^{\alpha_{i}}=\prod_{j=1}^{n_{i}}x_{j,i}^{\alpha_{j,i}}$,
$|\bvec{\alpha_{i}}| = \sum_{j=1}^{n_{i}}{\alpha_{j,i}}$, 
and $\xa^{\alpha}=\prod_{j=1}^{n_{i}}\xb_{i}^{\alpha_{i}}$.

An element $f$ of $R_{\delta}=R_{\delta_{1},\ldots,\delta_{k}}$ is represented as
$$ 
f
= \sum_{\bvec{\alpha}=(\bvec{\alpha}_{1},\ldots,\bvec{\alpha}_{k});
  |\bvec{\alpha}_{i}|\le \delta_{i} }
f_{\bvec{\alpha}} \, \xa_{1}^{\alpha_{1}}\cdots \xa_{k}^{\alpha_{k}}.
$$
The dimension of $R_{\delta}:=R_{\delta_{1},\ldots,\delta_{k}}$ is $n_{\delta_{1},\ldots,
  \delta_{k};n_{1},\ldots, n_{k}}$ $=$ $\prod _{i=1}^{k} {n_{i}+\delta_{i} \choose \delta_{i}}$.
For $\delta\in \NN, \alpha\in \NN^{n}$ with $|\alpha|\le \delta$, let 
${\delta\choose \alpha}={\delta! \over \alpha_{1}! \cdots
    \alpha_{n}! (\delta-|\alpha|)!}$.
We define the apolar inner product on $R_{\delta_{1},\ldots,\delta_{k}}$ by
$
  \<f|g\> = \sum_{|\alpha_{i}|\le \delta_{i}} f_{\alpha}\, g_{\alpha}
          {\delta_{1} \choose \alpha_{1}} \cdots {\delta_{k} \choose \alpha_{k}}.
$

The dual space of a $\kk$-vector space $E$
is denoted $\dual{E}=\Hom_{\kk}(E,\kk)$. It is the set of $\kk$-linear forms
from $E$ to $\kk$. 
A basis of the dual space $\dual{R}_\delta$,
is given by the set of linear forms that compute the
coefficients of a polynomial in the monomial basis $(\xa^{{\alpha}})_{\alpha
  \in \NN^{n_{1}}\times \cdots \times \NN^{n_{k}}; |\alpha_{i}|\leq \delta_{i}}$.  We
denote it by $(\db^{{\alpha}})_{\alpha\in \NN^{n_{1}}\times \cdots \times \NN^{n_{k}}; |\alpha_{i}|\leq \delta_{i}}$. 
We identify $\dRR$ with the (vector) space of formal power
series $\kk[[ \db ]]$ = $\kk[[ \db_{1},\ldots, \db_{k} ]]=
\kk[[d_{1,1}, \dots, d_{n_{1},1}$, $\ldots$, $d_{1,k}, \dots, d_{n_{k},k}]]$. 
Any element $\Lambda \in \dual{R}$ can be decomposed as 
$$ 
\Lambda = \sum_{\alpha\in \NN^{n_{1}}\times \cdots \times \NN^{n_{k}}} \,\Lambda (\xb^{\alpha}) \,\db^{\alpha}.
$$

Typical elements of $\dual{R}$ are the linear forms that correspond to the
evaluation at a point  $\zeta=(\zeta_{1},\ldots, \zeta_{i}) \in \kk^{n_{1}}\times \cdots \times \kk^{n_{k}}$:
\begin{displaymath}
  \begin{array}{lcl}
    \un_{\zeta }  & : & R \rightarrow \kk \\
    &  &  p \mapsto p (\zeta) \\
  \end{array}
\end{displaymath}
The decomposition of $\un_{\zeta}$ in the basis $\set{
  \bvec{d}^{\bvec{\alpha}}}_{\bvec{\alpha}\in \NN^{n_{1}}\times \cdots \times \NN^{n_{k}}}$
is 
$$
\un_{\zeta} 
= \sum_{\alpha \in \NN^{n_{1}}\times \cdots \times \NN^{n_{k}}} \zeta^{\alpha}\, \db^{\alpha}
= \sum_{\alpha \in \NN^{n_{1}}\times \cdots \times \NN^{n_{k}}} \prod_{i=1}^{k}\zeta_{i}^{ \alpha_{i}}\, \db_{i}^{\alpha_{i}}.
$$
We recall that the dual space $\dual{R}$ has a natural structure of $R$-module \cite{EM08} which is
defined as follows: for all $p \in R$, and for all $\Lambda \in \dual{R}$
consider the linear operator
\begin{displaymath}
  \begin{array}{lcl}
    p \star \Lambda & : & R \rightarrow \kk \\
    && q \mapsto \Lambda (p q).
  \end{array}
\end{displaymath}
In particular, we have $x_{i,j} \star\db_{1}^{\alpha_{1}}\cdots \db_{j}^{\alpha_{j}} \cdots
  \db_{k}^{\alpha_{k}} = $\\
{\small $\db_{1}^{\alpha_{1}}\cdots \db_{j-1}^{\alpha_{j-1}} \, d_{1,j}^{\alpha_{1,j}} \cdots d_{i-1,j}^{\alpha_{i-1,j}}
       d_{i,j}^{\alpha_{i,j} - 1} d_{i + 1,j}^{\alpha_{i+1,j}} \cdots
       d_{n_{j},j}^{\alpha_{n_{j},j}} \, \db_{j+1}^{\alpha_{j+1}}\cdots
       \db_{k}^{\alpha_{k}}$} if $\alpha_{i,j} > 0$ and $0$ otherwise.

\subsection{Tensor decomposition}
\label{sec:2}\label{sec:prob-formulation}\label{sec:def-poly-decomp}
In this section, we present different formulations of the tensor
decomposition problem, that we consider in this paper.

We will consider hereafter a partially symmetric tensor $T$ of 
$S^{\delta_{1}}(E_{1})\otimes \cdots \otimes S^{\delta_{k}}(E_{k})$ where
$E_{i}=\<x_{0,i},\ldots,x_{n_{i},i}\>$. It can be represented by a partially symmetric array of
coefficients 
\begin{equation}\label{array}
[T] = ( T_{\alpha_{1},\ldots,\alpha_{k}})_{\alpha_{i}\in \NN^{n_{i}+1}; |{\alpha}_{i}|=\delta_{i} }.
\end{equation}
For $\alpha_{i}\in \NN^{n_{i}}$ with $|\alpha_{i}|\le \delta_{i}$, we denote
$\overline{\alpha}_{i}=(\delta_{i}-|\alpha_{i}|,\alpha_{1,i},\ldots$, $\alpha_{n_{i},i})$
and, with an abuse of notation, we identify $T_{\alpha_{1},\ldots,\alpha_{k}} := T_{\overline{\alpha_{1}},\ldots,\overline{\alpha_{k}}}$.

Such a tensor is naturally associated with a (multihomogeneous) polynomial 
in the variables $\xb_{1}=(x_{0,1},\ldots, x_{n_{1},1})$, $\ldots$,
$\xb_{k}=(x_{0,k},\ldots, x_{n_{k},k})$ 
$$ 
T(\xb)=\sum_{
{{\alpha}=({\alpha}_{1},\ldots,{\alpha}_{k}) \in \NN^{n_{1}}\times \cdots \times \NN^{n_{k}};}
\atop {|{\alpha}_{i}|\leq\delta_{i}}
}
T_{{\alpha}} \, \xb_{1}^{{\overline{\alpha}_{1}}}\cdots \xb_{k}^{\overline{\alpha}_{k}}.
$$
or to an element $\underline{T}(\xa) \in R_{\delta_{1},\ldots,\delta_{k}}$
obtained by substituting $x_{0,i}$ by $1$ in $T(\xb)$ (for $i=1,\ldots, k$):
$$
\underline{T}(\xa)=\sum_{
\alpha \in \NN^{n_{1}}\times \cdots \times \NN^{n_{k}};
\atop
  |{\alpha}_{i}|\le \delta_{i}}
T_{{\alpha}} \, \xa_{1}^{\alpha_{1}}\cdots \xa_{k}^{\alpha_{k}}.
$$
An element of $R^{*}=\kk[[\db]]$ can also be associated naturally with $T$:
$$
{T^{*}}(\db)=\sum_{\alpha \in \NN^{n_{1}}\times \cdots \times \NN^{n_{k}};
\atop |{\alpha}_{i}|\le \delta_{i}}
{\delta_{1} \choose \alpha_{1}}^{-1} \cdots {\delta_{k} \choose \alpha_{k}}^{-1}
T_{{\alpha}} \, \db_{1}^{\alpha_{1}}\cdots \db_{k}^{\alpha_{k}}.
$$
so that for all $T'\in R_{\delta_{1},\ldots,\delta_{k}}$,
$$ 
\<T(\xa)|T'(\xa)\> = T^{*}(\db) (T'(\xa)).
$$

The decomposition of  tensor $T$ can be stated as follows:
\par\textbf{Tensor decomposition problem.}
{\em Given $T(\xb) \in S^{\delta_{1}}(E_{1})$ $\otimes \cdots \otimes
S^{\delta_{k}}(E_{k})$, find a decomposition of $T(\xb)$ as a sum of products
of powers of linear forms in $\xb_{j}$:
\begin{equation}\label{eq:poly-decomp}
T(\xb) = \sum_{i=1}^{r}\, \gamma_{i}\, \lb_{1,i}(\xb_{1})^{\delta_{1}} \cdots \lb_{k,i}(\xb_{k})^{\delta_{k}} 
\end{equation}
where $\gamma_{i}\neq 0$, $\lb_{j,i}(\xb_{j})= l_{0,j,i} x_{0,j} + l_{1,j,i} x_{1,j} + \cdots + l_{ n_{j},j,i} x_{j,n_{j}}$
and $r$ is the smallest possible integer for such a decomposition.}
\begin{definition}\label{rank}
The minimal number of terms $r$ in a decomposition of the form \eqref{eq:poly-decomp} is called the {\em rank} of $T$.
\end{definition}
 
We say that $T(\xb)$ has an {\em affine decomposition} if there exists a minimal
decomposition  of $T(\xb)$ of the form \eqref{eq:poly-decomp} where $r$ is
the rank of $T$ and such that $l_{0,j,i}\neq 0$ for $i=1,\ldots,r$. 
Notice that by a generic change of coordinates in $E_{i}$, we may assume
that all  $l_{0,j,i}\neq 0$ and thus that $T$ has an affine decomposition.
Suppose that $T(\xb)$ has an affine decomposition. Then by scaling
$\lb_{j,i}(\xb_{j})$ and multiplying $\gamma_{i}$ by the inverse of the
$\delta_{j}^{th}$ power of this scaling factor, we may assume that
$l_{0,j,i}=1$.  Thus, the polynomial
$$
\underline{T}(\xa) = \sum_{i=1}^{r}\, \gamma_{i}\, \sum_{|\alpha_{i}|\le \delta_{i}} \,
{\delta_{1} \choose \alpha_{1}} \cdots {\delta_{k} \choose \alpha_{k}} 
\, \zeta_{1,i}^{\alpha_{1}} \cdots \zeta_{k,i}^{\alpha_{k}}\, \xa_{1}^{\alpha_{1}} \cdots \xa_{k}^{\alpha_{k}}
$$
with $T_{\alpha_{1},\ldots,\alpha_{k}}= \sum_{i=1}^{r}\, \gamma_{i}\, \sum_{|\alpha_{i}|\le \delta_{i}}
{\delta_{1} \choose \alpha_{1}} \cdots {\delta_{k} \choose \alpha_{k}} 
\, \zeta_{1,i}^{\alpha_{1}} \cdots \zeta_{k,i}^{\alpha_{k}}$. Equivalently, we have
$$
T^{*}(\db) = \sum_{i=1}^{r}\, \gamma_{i}\, \sum_{|\alpha_{i}|\le \delta_{i}}
\, \zeta_{1,i}^{\alpha_{1}} \cdots \zeta_{k,i}^{\alpha_{k}}\, \db_{1}^{\alpha_{1}} \cdots \db_{k}^{\alpha_{k}}
$$
so that $T^{*}(\db)$ coincides on $R_{\delta_{1},\ldots,\delta_{k}}$
with the linear form 
$$ 
\sum_{i=1}^{r}\, \gamma_{i} \, \un_{\zeta_{1,i},\ldots, \zeta_{k,i}} = \sum_{i=1}^{r}\, \gamma_{i} \, \un_{\zeta_{i}}
$$
with $\zeta_{i} := (\zeta_{1,i},\ldots, \zeta_{k,i}) \in \kk^{n_{1}}\times
\cdots \kk^{n_{k}}$.

The decomposition problem  can then be restated as follows:
\par\textbf{Interpolation problem.} {\em Given $T^{*} \in R_{\delta_{1},\ldots,\delta_{k}}^{*}$ which admits an
   affine decomposition, find the minimal number
  of non-zero vectors
  $\zeta_{1}, \ldots, \zeta_{r}\in \kk^{n_{1}} \times \cdots \times \kk^{n_{k}}$ and non-zero scalars $\gamma_{1},
  \ldots, \gamma_{r} \in \kk-\{0\}$ such that 
  \begin{equation}\label{interpol}
T^{*} = \sum_{i=1}^{r} \gamma_{i} \, \un_{\zeta_{i}} 
  \end{equation}
on $R_{\delta_{1},\ldots,\delta_{k}}$.}

If such a decomposition exists, we say that $\Lambda = \sum_{i=1}^{r}
\gamma_{i} \, \un_{\zeta_{i}}\in R^{*}$ extends $T^{*} \in
R_{\delta_{1},\ldots, \delta_{k}}^{*}$.


\subsection{Decomposable tensors}\label{sec:geometric}

In this section, we analyze the set of  tensors of
rank $1$, also called \emph{decomposable} tensors \cite{AboOP09:tams}. They naturally form projective varieties, which we are going to
describe using the language of projective geometry.

We begin by defining two auxiliary but very classical varieties, namely Segre and Veronese varieties.
\begin{definition} The image of the following map
$$\begin{array}{rcl}
s_k: \Proj(E_1) \times \cdots \times \Proj (E_k) &\rightarrow & \Proj(E_1 \otimes \cdots \otimes E_k) \\
 (\overline{\bvec{v}_1}, \ldots ,\overline{\bvec{v}_k}) &\mapsto &\overline{\bvec{v}_1 \otimes \cdots \otimes \bvec{v}_k}
\end{array}$$
is the so called Segre variety of $k$ factors. We denote it by $\SV(E_{1}\otimes \cdots \otimes E_{k})$.
 \end{definition}


From Definition \ref{rank} of the rank of a tensor and from the Interpolation Problem point of view (\ref{interpol}) we see that a Segre variety parametrizes projective classes of rank 1 tensors $T=\bvec{v_1}\otimes \cdots \otimes \bvec{v_k}\in E_1 \otimes \cdots \otimes E_k$ for certain $\bvec{v_i} \in E_i$, $i=1, \ldots ,k$.

\begin{definition}\label{flat}  Let $(J_{1},J_{2})$ be a partition of the set $\{1, \ldots , k\}$. If $J_{1}=\{h_{1}, \ldots , h_{s}\}$ and $J_{2}=\{1, \ldots ,k\} \backslash J_{1}=\{h'_{1}, \ldots , h'_{k-s}\}$, the $(J_{1},J_{2})$-Flattening of $E_{1}\otimes \cdots \otimes E_{k}$  is the following:
$$E_{J_{1}}\otimes E_{J_{2}}=(E_{h_{1}}\otimes \cdots \otimes E_{h_{s}})\otimes (E_{h'_{1}}\otimes \cdots \otimes E_{h'_{k-s}}).$$
\end{definition}

Let $E_{J_{1}}\otimes E_{J_{2}}$ be any flattening of $E_{1}\otimes \cdots
\otimes E_{k}$ as in Definition \ref{flat} and let
$f_{J_{1},J_{2}}:\Proj(E_{1}\otimes \cdots \otimes E_{k})\rightarrow
\Proj(E_{J_{1}}\otimes E_{J_{2}})$ be the obvious isomorphism. Let $[T]$ be
an array associated with a tensor $T\in E_{1}\otimes \cdots \otimes E_{k}$; let
$\overline{T'}=f_{J_{1},J_{2}}(\overline{\bvec{T}})\in \Proj(E_{J_{1}}\otimes
E_{J_{2}})$ and let $[A_{J_{1},J_{2}}]$ be the matrix associated with
$T'$. Then the $d$-minors of the matrix $[A_{J_{1},J_{2}}]$ are said to be
$d$-minors of $[T]$. 

An array $[A] =(x_{i_{1}, \ldots , i_{k}})_{0\leq i_{j}\leq n_{j}\, , \, j=1,
  \ldots , k}$ is said to be a generic array of indeterminates  of
$R=\kk[\xa_{1},\ldots$, $\xa_{k}]$  if the entries of $[A]$ are the independent
variables of $R$. 

It is a classical result due to R. Grone (see \cite{Gr77}) that a set of equations for a  Segre variety is given by all the $2$-minors of a generic array. 
In \cite{TH02} it is proved that, if $[A]$ is a generic array in $R$ of size $(n_1+1)\times \cdots \times (n_k+1)$ and $I_{d}([A])$ is the ideal generated by the $d$-minors of $[A]$ , then $I_{2}([A])$ is a prime ideal, therefore:
$$I(\SV(E_{1}\otimes \cdots \otimes E_{k}))=I_{2}([A]).$$

We introduce now the Veronese variety. Classically it is defined to be the $d$-tuple embedding of $\Proj^n$ into $\Proj^{{n+d \choose d}-1}$ via the linear system associated with the sheaf $\mathcal{O}(d)$ with $d>0$. We give here an equivalent definition.

Let $E$ be an $n+1$ dimensional vector space. With the notation $S^d(E)$ we mean the vector subspace of $E^{\otimes d}$ of symmetric tensors.
 
\vspace{-0.18cm}\begin{definition}\label{veronesedef}
 The image of the following map 
$$\begin{array}{rcl}
 \nu_d:\Proj (E)&\rightarrow& \Proj (S^d(E))\\
\overline{\bvec{v}}&\mapsto & \overline{\bvec{v^{\otimes d}} }
\end{array}
$$
is the so called Veronese variety. We indicate it with $\SV(S^d(E))$.
\end{definition}

With this definition it is easy to see that the Veronese variety parametrizes symmetric rank 1 tensors.

Observe that if we take the vector space $E$ to be a vector space of linear forms $\langle x_0 , \ldots , x_n \rangle$  then the image of the map $\nu_d$ above parametrizes homogeneous polynomials that can be written as $d$-th powers of linear forms.

The Veronese variety $\SV(S^d(E))\subset \Proj (S^d(E))$ can be also viewed as 
$\SV(S^d(E))=\SV(E^{\otimes d})\cap \Proj(S^d(E))$.

Let $[A]=(x_{i_{1}, \ldots  , i_{d}})_{0\leq i_{j}\leq n, \, j=1, \ldots , d} $ be a generic symmetric array. It is a known result   that:
\begin{equation}\label{veronese}I(\SV(S^d(E)))=I_{2}([A]).\end{equation}
See \cite{Wa} for the set theoretical point of view. In \cite{Pu} the author
proved that $I(\SV(S^d(E)))$ is generated by the $2$-minors of a particular
catalecticant matrix (for a definition of ``Catalecticant matrices'' see e.g.
either \cite{Pu} or \cite{Ge99}). A. Parolin, in his PhD thesis
(\cite{PaPhd}), proved that the ideal generated by the $2$-minors of that
catalecticant matrix is actually $I_{2}([A])$. 

We are now ready to describe the geometric object that parametrizes partially symmetric tensors $T\in S^{\delta_1} (E_1)\otimes \cdots \otimes S^{\delta_k}(E_k)$. Let us start with the rank 1 partially symmetric tensors.

\begin{definition}\label{segrevero} Let $E_{1}, \ldots , E_{k}$ be vector spaces
of dimensions $n_{1}+1, \ldots , n_{k}+1$ respectively. 
The Segre-Veronese variety $\SV(S^{\delta_{1}}(E_{1})\otimes \cdots \otimes S^{\delta_k}(E_{k}))$ is the embedding of $\Proj(E_{1})\otimes \cdots \otimes \Proj(E_{k})$ into $\Proj^{N-1}\simeq \Proj(S^{\delta_{1}}(E_{1})\otimes \cdots \otimes S^{\delta_k}(E_{k}))$, where $N=\left(\Pi_{i=1}^{k}{n_{i}+\delta_{i}\choose d_{i}}\right)$, given by sections of the sheaf ${\mathcal{O}}(\delta_{1}, \ldots , \delta_{k})$. 
\\
I.e. $\SV(S^{\delta_{1}}(E_{1})\otimes \cdots \otimes S^{\delta_k}(E_{k}))$ is the image of the composition of the following two maps:
$$
\Proj(E_{1})\times \cdots \times \Proj(E_{k}) \stackrel{\nu_{\delta_{1}}\times \cdots \times \nu_{\delta_{k}}}{\longrightarrow} \Proj^{{n_{1}+\delta_{1}\choose \delta_{1}}-1}\times \cdots \times \Proj^{{n_{k}+\delta_{k}\choose \delta_{k}}-1}$$
and $ \Proj^{{n_{1}+\delta_{1}\choose \delta_{1}}-1}\times \cdots \times \Proj^{{n_{k}+\delta_{k}\choose \delta_{t}}-1}\stackrel{s}{\longrightarrow} \Proj^{N-1}$,
where each $\nu_{\delta_i}$ is a Veronese embedding of $\mathbb{P}(E_i)$ as in Definition \ref{veronesedef}, then $Im (\nu_{\delta_{1}}\times \cdots \times \nu_{\delta_{k}})=\SV(S^{\delta_{1}}(E_1))\times \cdots \times \SV(S^{\delta_k}(E_k))$ and $Im (s)$ is the Segre variety of $k$ factors. Therefore the Segre-Veronese variety is the Segre re-embedding of the product of $k$ Veronese varieties.
\end{definition}

If $(\delta_{1}, \ldots ,\delta_{k})=(1, \ldots , 1)$ then the corresponding
Segre-Veronese variety is nothing else than the classical Segre variety of $\Proj(E_{1}\otimes \cdots \otimes E_{k})$.

If $k=1$ then the corresponding Segre-Veronese variety is nothing else than
the classical Veronese variety of $\Proj(S^{\delta_1}(E_1))$.

Observe that $\SV(S^{\delta_{1}}(E_{1})\otimes \cdots \otimes S^{\delta_k}(E_{k}))$ can be viewed as the intersection with the Segre variety $\SV(E_1^{\otimes \delta_1}\otimes \cdots \otimes E_k^{\otimes \delta_k})$ that parametrizes rank one tensors  and the projective subspace $\Proj(S^{\delta_1}(E_1) \otimes \cdots \otimes S^{\delta_k}(E_k))\subset \Proj (E_1^{\otimes \delta_1}\otimes \cdots \otimes E_k^{\otimes \delta_k})$ that parametrizes partially symmetric tensors:
$\SV(S^{\delta_1}(E_{1})\otimes \cdots \otimes
S^{\delta_k}(E_{k}))=\SV(E_1^{\otimes \delta_1}\otimes \cdots \otimes
E_k^{\otimes \delta_k})$ $\cap$ $\Proj(S^{\delta_1}(E_1) \otimes \cdots \otimes S^{\delta_k}(E_k))$.

In \cite{Be08} it is proved that  if $[A]$ is a generic array of indeterminates associated with the  multihomogeneous polynomial ring  $S^{\delta_1}(E_1) \otimes \cdots \otimes S^{\delta_k}(E_k)$  
(i.e. it is a generic partially symmetric array),
the ideal of the Segre-Veronese variety $\SV(S^{\delta_{1}}(E_{1})\otimes \cdots \otimes S^{\delta_k}(E_{k}))$ is 
$$I(\SV(S^{\delta_{1}}(E_{1})\otimes \cdots \otimes S^{\delta_k}(E_{k})))=I_{2}([A])$$
with $\delta_{i}>0$ for $i=1, \ldots, k$.

Now if we consider the vector spaces of linear forms  $E_i\simeq S^1(E_i)$ for $i=1, \ldots ,k$, we get that the Segre-Veronese variety $\SV(S^{\delta_{1}}(E_{1})\otimes \cdots \otimes S^{\delta_k}(E_{k}))$ parametrizes multihomogenoeus polynomials $F\in S^{\delta_1}(E_1) \otimes \cdots \otimes S^{\delta_k}(E_k)$ of the type $F=\lb_{1}^{\delta_{1}} \cdots \lb_{k}^{\delta_{k}} $
where $\lb_{i}$ are linear forms in $S^1(E_i)$ for $i=1, \ldots , k$.
\\
\indent From this observation we understand that the tensor decomposition problem of finding a minimal decomposition of type (\ref{eq:poly-decomp}) for an element $T \in S^{\delta_{1}}(E_{1})\otimes \cdots \otimes S^{\delta_{k}}(E_{k})$ is equivalent to finding the  minimum number of elements belonging to the Segre-Veronese variety $\SV(S^{\delta_{1}}(E_{1})\otimes \cdots \otimes S^{\delta_k}(E_{k}))$ whose span contains $\overline{T}\in\Proj ( S^{\delta_{1}}(E_{1})\otimes \cdots \otimes S^{\delta_{k}}(E_{k}))$.

The natural geometric objects that are associated with this kind of problems
are the higher secant varieties of the  Segre-Veronese varieties that we are
going to define. 

\begin{definition}
 Let $X\subset \Proj^N$ be any projective variety and define 
$$ 
X^0_s:= \bigcup_{\overline{\bvec{P}}_{1}, \ldots
  ,\overline{\bvec{P}}_{s}\in X}\langle \overline{\bvec{P}}_{1} , \ldots ,
\overline{\bvec{P}}_{s} \rangle.
$$
The $s$-th secant variety $X_s\subset \Proj^N$ of $X$ is the Zariski closure of $X^0_s$.
\end{definition}

Observe that the generic element of $X_s$ is a point $\overline{\bvec{P}}\in
\Proj^N$ that can be written as a linear combination of $s$ points of $X$, in
fact a generic element of $X_s$ is an element of $X^0_s$. Therefore if $X$ is
the Segre-Veronese variety, then the generic element of
$\SV_s(S^{\delta_{1}}(E_{1})\otimes \cdots \otimes S^{\delta_k}(E_{k}))$ is
the projective class of a partially symmetric tensor $T\in
S^{\delta_1}(E_1)\otimes \cdots \otimes S^{\delta_k}(E_k)$ that can be
written as a linear combination of $s$ linearly independent partially
symmetric tensors of rank 1. Unfortunately not all the elements of
$\SV_s(S^{\delta_{1}}(E_{1})\otimes \cdots \otimes S^{\delta_k}(E_{k}))$ are
of this form. In fact if $\overline{T}\in \SV_s(S^{\delta_{1}}(E_{1})\otimes
\cdots \otimes S^{\delta_k}(E_{k}))\setminus
\SV_s^0(S^{\delta_{1}}(E_{1})\otimes \cdots \otimes S^{\delta_k}(E_{k}))$
then the rank of $T$ is strictly bigger than $s$.

\begin{definition}
 The minimum integer $s$ such that $\overline{T} \in \Proj(S^{\delta_{1}}(E_{1})\otimes \cdots \otimes
S^{\delta_{k}}(E_{k}))$ belongs to  $\SV_s(S^{\delta_{1}}(E_{1})\otimes \cdots \otimes S^{\delta_k}(E_{k}))$ is called the  border rank of $T$.
\end{definition}

In order to find the border rank of a tensor $T\in
S^{\delta_{1}}(E_{1})\otimes \cdots \otimes S^{\delta_{k}}(E_{k})$ we should
need a set of equations for $\SV_s(S^{\delta_{1}}(E_{1})$ $\otimes \cdots
\otimes S^{\delta_k}(E_{k}))$ for $s>1$. The knowledge of the generators of
the ideals of secant varieties of homogeneous varieties is a very deep
problem that is solved only in very particular cases (see eg.
\cite{Ott07}, \cite{LandW07}, \cite{LandM08}, \cite{LOarxiv10first},
\cite{BGLarxiv10}, \cite{LOarxiv10}).  

From a computational point of view, there is a very direct and well known way of
getting the equations for the secant variety, which consists of introducing
parameters or unknowns for the coefficients of $\lb_{i,j}$ and $\gamma_{i}$ in
\eqref{eq:poly-decomp}, to expand the polynomial and identify its coefficients with the
coefficients of $T$. Eliminating the coefficients of $\lb_{i,j}$ and
$\gamma_{i}$ yields the equations of the secant variety.

Unfortunately this procedure is far from being computationally practical,
because we have to deal with high degree polynomials in many variables, with
a lot of symmetries. This is why we need to introduce moment matrices and to
use a different kind of elimination.



\subsection{Moment matrices}
\label{sec:4}\label{sec:moments}

In this section, we recall the algebraic tools and the properties we  need to describe and
analyze our algorithm. Refer e.g. to \cite{BracCMT09:laa}, \cite{EM08}, \cite{mp-jcomplexity-2000} for more details.
 
\label{sec:hankel-op}
\label{sec:quadratic-form}

For any $\Lambda \in \dual{R}$, 
define the bilinear form $Q_{\Lambda}$, such that 
$\forall a,b \in R$,  $Q_{\Lambda}(a, b)=\Lambda (a b)$.
The matrix of $Q_{\Lambda}$ in the monomial basis, 
of $R$ is 
$\mat{Q}_{\Lambda} = (\Lambda (\xb^{\alpha + \beta}))_{\alpha, \beta}$,
where $\alpha, \beta \in \NN^n$.
Similarly, for any $\Lambda \in \dual{R}$, we define the Hankel operator $H_{\Lambda}$
from $R$ to $\dual{R}$ as 
\begin{displaymath}
  \begin{array}{lcl}
    H_{\Lambda} & : & R \rightarrow \dRR  \\
    && p \mapsto p \star \Lambda.
  \end{array}
\end{displaymath}
The matrix of the linear operator $ H_{\Lambda}$ 
in the monomial basis, 
and in the dual basis, $\set{ \bvec{d}^{\alpha}}$,
is $\mat{H}_{\Lambda} = ( \Lambda( \xb^{\alpha + \beta}))_{\alpha, \beta}$,
where $\alpha, \beta \in \NN^n$.
The following relates the Hankel operators with the bilinear forms.
For all $a, b \in R$, thanks to the $R$-module structure, it holds 
\begin{displaymath}
  Q_{\Lambda} (a, b) = \Lambda( a b) = a  \star \Lambda (b) 
= H_{\Lambda} (a) (b) 
.
\end{displaymath}
In what follows, we will identify $H_{\Lambda}$ and $Q_{\Lambda}$.
\begin{definition} 
  Given $B=\{b_{1},\ldots,b_{r}\}, B' =\{b'_{1},\ldots$, $b'_{r'}\}\subset R$, we define
  $$ 
  H^{B,B'}_{\Lambda} : \<B\> \rightarrow \dual{ \<B'\>},
  $$
  as the restriction of $H_{\Lambda}$ to the vector space $\<B\>$ and
inclusion of $R^{*}$ in $\dual{ \<B'\>}$. Let $\mat{H}^{B,B'}_{\Lambda} = 
  (\Lambda(b_{i}\, b'_{j}))_{1 \le i\le r, 1\le j\le r'}$. If  $B'=B$, we
also use the notation $H^{B}_{\Lambda}$ and $\mat{H}^{B}_{\Lambda}$.
\end{definition}
If $B, B'$ are linearly independent, then
$\mat{H}^{B,B'}_{\Lambda}$ is the matrix of
$H^{B,B'}_{\Lambda}$ in this basis $\{b_{1},\ldots,b_{r}\}$ of $\<B\>$ and the
dual basis of $B'$ in $\dual{ \<B'\>}$.
The {\em catalecticant} matrices of \cite{ia-book-1999} correspond to the case where
$B$ and $B'$ are respectively the set of monomials of degree $\le k$ and $\le
d-k$ ($k=0,\ldots,d$).

From the definition of the Hankel operators, we can deduce that a polynomial $p \in R$
belongs to the kernel of $\mat{H}_{\Lambda}$ if and only if
$p \star \Lambda = 0$, which in turn holds if and only if 
for all $q \in R$, $\Lambda(p q) = 0$.

\begin{proposition}
  \label{prop:kernel-is-ideal}
  Let $I_{\Lambda}$ be the kernel of ${H}_{\Lambda}$.
  Then, $I_{\Lambda}$ is an ideal of $R$.
\end{proposition}
\begin{proof}
  Let $p_1, p_2 \in I_{\Lambda}$. Then for all $q \in R$,
  $\Lambda( (p_1+p_2)q) = \Lambda(p_1 q) + \Lambda( p_2 q) = 0$.
  Thus, $p_1 + p_2 \in I_{\Lambda}$.
   
  If $p \in I_{\Lambda}$ and $p' \in R$,
  then for all $q \in R$, it holds $\Lambda (p p' q) = 0$. 
  Thus $p p' \in I_{\Lambda}$ 
  and $I_{\Lambda}$ is an ideal.
\end{proof}

Let $\AA = R / I_{\Lambda}$ be the quotient algebra of polynomials 
modulo the ideal $I_{\Lambda}$, which, as  Proposition~\ref{prop:kernel-is-ideal} 
states is the kernel of ${H_{\Lambda}}$. The rank of $H_{\Lambda}$  is the
dimension of $\AA$ as a $\kk$-vector space.

\begin{definition}
For any $B\subset R$, let $B^{+} = B\cup x_{1}B \cup \cdots \cup x_{n}B$ and
$\partial B= B^{+}\setminus B$.
\end{definition}
 
\begin{proposition}\label{prop:basis:ideal}
  Assume that $\rank( {H}_{\Lambda}) = r < \infty$ 
  and let $B=\{b_1, \dots,  b_r\} \subset R$ such that 
  $\mat{H}_{\Lambda}^B$ is invertible. 
  Then $b_1, \dots, b_r$ is a basis of $\AA$. If $1 \in \<B\>$ 
the ideal $I_{\Lambda}$
is generated by $\ker H^{B^{+}}_{\Lambda}$.
\end{proposition}
\begin{proof}
  Let us first prove that $\Set{b_1, \ldots, b_r} \cap I_{\Lambda} =\{0\}$.
  Let $p \in \langle b_1, \ldots, b_r \rangle \cap I_{\Lambda}.$ 
  Then $p =  \sum_i{ p_i \,b_i}$ with $p_{i} \in \kk$ and $\Lambda (p \, b_j) = 0$.
  The second equation implies that $\mat{H}_{\Lambda}^B \cdot \bvec{p} = \bvec{0}$, 
  where $\bvec{p}=[p_1, \dots, p_r]^{t}\in \kk^{r}$. 
  Since $\mat{H}_{\Lambda}^B$ is invertible, this
  implies that $\bvec{p} = \bvec{0}$ and $p = 0$.

  As a consequence, we deduce that 
  $b_1 \star \Lambda, \dots, b_r \star \Lambda$ 
  are linearly independent elements of $\dual{R}$. 
  This is so, because otherwise there exists
  $\bvec{m} = [\mu_1, \ldots, \mu_r]^{\top} \neq \bvec{0}$, 
  such that 
  $\mu_1 (b_1 \star \Lambda) + \dots  + \mu_r (b_r \star \Lambda) = 
  (\mu_1 b_1 + \cdots + \mu_r b_r) \star \Lambda = 0$. 
  As $\Set{b_1, \ldots, b_r} \cap \mathrm{Ker}( \mat{H}_{\Lambda}) = \set{ 0}$, this
  yields a contradiction.

  Consequently, $\set{b_1 \star \Lambda, \ldots, b_r \star \Lambda}$ span the image
  of ${H}_{\Lambda}$. For any $p \in R$, it holds that   
  $p \star \Lambda = \sum_{i = 1}^r{\mu_i  (b_i \star \Lambda)}$ 
  for some $\mu_1, \ldots, \mu_r \in \kk$. We
  deduce that $p - \sum_{i = 1}^r \mu_i b_i \in I_{\Lambda}$. This yields the
  decomposition 
   $R = \langle b_1, \ldots, b_r \rangle \oplus K$,
  $R = B \oplus I_{\Lambda}$,
  and shows that
  $b_1, \ldots, b_r$ is a basis of $\AA$.

If $1\in \<B\>$, the ideal $I_{\Lambda}$ is generated by the relations $x_{j}
b_{k} - \sum_{i = 1}^r \mu_i^{j,k} b_i \in I_{\Lambda}$. These are precisely
in the kernel of $H_{\Lambda}^{B^{+}}$.
\end{proof}

\begin{proposition}
  \label{prop:form-decomp}
  If $\rank( {H}_{\Lambda}) = r < \infty$, 
  then $\AA$ is of dimension $r$ over $\kk$ and there exist
  $\zeta_1, \ldots, \zeta_d \in \kk^n$,
  where $d \leq r$, and $p_i \in \kk[ \partial_1, \dots, \partial_n]$, such
  that
  \begin{equation}
    \Lambda = \sum_{i = 1}^d \un_{\zeta_i} \circ p_i ( \bvec{\partial})
    \label{eq:lambda}
  \end{equation}
  Moreover the multiplicity of $\zeta_{i}$ is the dimension of the vector space
  spanned the inverse system generated by $\un_{\zeta_i} \circ p_i ( \bvec{\partial})$.
\end{proposition}
\begin{proof}
  Since $\rank( \mat{H}_{\Lambda}) = r$, 
  the dimension of the vector space $\mathcal{A}_{\Lambda}$ is also $r$.
  Thus the number of zeros of the ideal $I_{\Lambda}$,
  say $\{\zeta_1, \ldots, \zeta_d \}$ is at most $r$, viz. $d \leq r$. 
  We can apply the structure Theorem \cite[Th.~7.34, p. 185]{EM08}
  in order to get the decomposition.
\end{proof}
In characteristic $0$, the inverse system of  $\un_{\zeta_i} \circ p_i (
\bvec{\partial})$ by $p_{i}$ is isomorphic to the vector space generated by
$p_{i}$ and its derivatives of any order with respect to the variables $\partial_{i}$.
In general characteristic, we replace the derivatives by the product by the
``inverse'' of the variables \cite{mp-jcomplexity-2000}, \cite{EM08}.

\begin{definition} For $T^{*}\in R_{\delta_{1},\ldots,\delta_{k}}^{*}$, we call generalized decomposition of ${T^{*}}$
a decomposition such that ${T^{*}} = \sum_{i = 1}^d \un_{\zeta_i}
\circ p_i ( \bvec{\partial})$ where the sum for $i=1,\ldots,d$ of the dimensions of the vector spaces 
  spanned by the inverse system generated by $\un_{\zeta_i} \circ
p_i ( \bvec{\partial})$ is minimal. This minimal sum of dimensions is called the length of $f$.
\end{definition}
This definition extends the definition introduced in \cite{ia-book-1999} for
binary forms. The length of $\dual{T}$ is the rank of the corresponding
Hankel operator $H_{\Lambda}$.

\begin{theorem} \label{prop:decomp-rank}
Let $\Lambda \in \dual{R}$ such that $\Lambda = \sum_{i = 1}^r \gamma_{i}\, \un_{\zeta_i}$
  with $\gamma_{i}\neq 0$ and $\zeta_{i}$ distinct points of $\kk^{n}$, iff
$\rank H_{\Lambda}=r$ and $I_{\Lambda}$ is a radical ideal.
\end{theorem}
\begin{proof}
If $\Lambda = \sum_{i = 1}^r \gamma_{i}\, \un_{\zeta_i}$, with
$\gamma_{i}\neq 0$ and $\zeta_{i}$ distinct points of $\kk^{n}$.
Let $\{e_{1},\ldots,e_{r}\}$ be a family of interpolation polynomials at these
points: $e_{i}(\zeta_{j})=1$ if $i=j$ and $0$ otherwise.
Let $I_{\zeta}$ be the ideal of polynomials which vanish at
$\zeta_{1},\ldots,\zeta_{r}$. It is a radical ideal. 
We have clearly $I_{\zeta}\subset I_{\Lambda}$. For any $p \in I_{\Lambda}$,
and $i=1,\ldots,r$, 
we have $p\star\Lambda(e_{i})= \Lambda(p\, e_{i}) = p(\zeta_{i})=0$, which
proves that $I_{\Lambda} = I_{\zeta}$ is a radical ideal.
As the quotient $\AA$ is generated by the interpolation polynomials
$e_{1},\ldots,e_{r}$, $H_{\Lambda}$ is of rank $r$.

Conversely, if $\rank H_{\Lambda}=r$, by Proposition \ref{prop:form-decomp}
$ \Lambda = \sum_{i = 1}^r \un_{\zeta_i} \circ p_i ( \bvec{\partial})$
with a polynomial of degree $0$, since the multiplicity of $\zeta_{i}$ is
$1$. This concludes the proof of the equivalence.
\end{proof}
In the binary case this also corresponds to the border rank of $T^*$,
therefore the $r$-th minors of the Hankel operator give equations for the
$r$-th secant variety to the rational normal curves \cite{ia-book-1999}.

In order to compute the zeroes of an ideal $I_{\Lambda}$ when we know a basis
of $\AA$, we exploit the properties of the operators of multiplication in
$\AA$:
$M_{a} : \AA \rightarrow \AA$, such that
$\forall b \in \AA, M_{a}(b)= a\, b$ and its transposed operator $M_{a}^{t} :
\dAA \rightarrow \dAA$, such that for 
$\forall  \gamma \in \dAA,  M_{a}^{\top}(\gamma) = a\star \gamma$.

The following proposition expresses a similar result, based on the properties
of the duality. 
\begin{proposition}
  \label{prop:Hl-mx}
  For any linear form $\Lambda \in \dual{R}$ such that 
  $\rank H_{\Lambda} < \infty$ and any $a \in \AA$, we have
  \begin{equation}
    H_{a \star \Lambda} = M_a^{t} \circ H_{\Lambda} \label{eq:multa}
  \end{equation} 
\end{proposition}
\begin{proof}
  By definition, $\forall p \in R, H_{a \star \Lambda} (p) = a\, p\star \Lambda
  = a\star  (p\star \Lambda) = M_{a}^{\top} \circ H_{\Lambda} (p)$.
\end{proof}
We have the following well-known theorem: 
\begin{theorem}[\cite{CLO2,CLO,EM08}]
  \label{th:stickelberger}
  Assume that $\AA$ is a finite dimensional vector space. Then 
  $\Lambda = \sum_{i = 1}^d \un_{\zeta_i} \circ p_i ( \bvec{\partial})$
  for $\zeta_{i}\in \kk^{n}$ and 
  $p_i ( \partial) \in \kk[ \partial_1, \dots, \partial_n]$ and 
\par $\bullet$  the eigenvalues of the operators ${M}_a$ and ${M}^{t}_a$,  
    are given by $\Set{ a( \zeta_1), \dots, a( \zeta_r)}$.
\par $\bullet$ the common eigenvectors of the operators $({M}_{x_i}^{t})_{1 \leq i \leq n}$ are
    (up to scalar) $\un_{\zeta_{i}}$.
\end{theorem}
Using the previous proposition, one can recover the points  $\zeta_{i}\in
\kk^{n}$ by eigenvector computation as follows.
Assume that $B\subset R$ with $|B|=\rank( H_{\Lambda})$, then equation \eqref{eq:multa}
and its transposition yield
$$ 
\mat{H}^{B}_{a \star \Lambda} = \mat{M}_{a}^{t} \mat{H}^{B}_{\Lambda} =
\mat{H}^{B}_{\Lambda} \, \mat{M}_{a},
$$  
where $\mat{M}_{a}$ is the matrix of multiplication by $a$ in the basis $B$
of $\AA$. By Theorem \ref{th:stickelberger}, the common solutions of the generalized
eigenvalue problem 
\begin{equation}\label{eq:geiv}
  (\mat{H}_{a \star \Lambda} - \lambda \, \mat{H}_{\Lambda}) \bvec{v} = \mat{O}
\end{equation}
for all $a\in R$, 
yield the common eigenvectors $\mat{H}^{B}_{\Lambda} \bvec{v}$ of
$\mat{M}_a^{t}$, that is the evaluation $\un_{\zeta_{i}}$ at the roots. 
Therefore, these common eigenvectors $\mat{H}^{B}_{\Lambda} \bvec{v}$ are up to a
scalar, the vectors $[b_{1}(\zeta_{i}),\ldots,b_{r}(\zeta_{i})]$
$(i=1,\ldots,r)$.
Notice that it is sufficient to compute the common eigenvectors of 
 $(\mat{H}_{x_{i} \star \Lambda} , \mat{H}_{\Lambda})$ for $i=1,\ldots, n$

If $\Lambda = \sum_{i = 1}^d \gamma_{i} \un_{\zeta_i}$ $(\gamma_{i}\neq
0)$, then the roots are simple, and one eigenvector computation is enough:
for any $a\in R$, $\mat{M}_{a}$ is diagonalizable and the
generalized eigenvectors $\mat{H}^{B}_{\Lambda} \bvec{v}$ are, up to a scalar,
the evaluation $\un_{\zeta_{i}}$ at the roots.

Coming back to our problem of partially symmetric tensor decomposition,
$T^{*}\in {R^{*}_{\delta_{1},\ldots,\delta_{k}}}$ admits an affine
decomposition of rank $r$ iff $T^{*}$ coincide on
$R_{\delta_{1},\ldots,\delta_{k}}$ with 
$$ 
\Lambda = \sum_{i=1}^{r} \gamma_{i}\, \un_{\zeta_{i}},
$$ 
for some distinct $\zeta_{1},\ldots, \zeta_{r} \in \kk^{n_{1}}\times \cdots
\times \kk^{n_{k}}$ and some $\gamma_{i}\in \kk-\{0\}$. 
Then, by theorem \ref{prop:decomp-rank},
 $H_{\Lambda}$ is of rank $r$ and $I_{\Lambda}$ is radical. 

Conversely, given $H_{\Lambda}$ of rank $r$ with  $I_{\Lambda}$ radical
which coincides on $R_{\delta_{1},\ldots,\delta_{k}}$ with $T^{*}$, by proposition
\ref{prop:form-decomp}, $\Lambda= \sum_{i=1}^{r} \gamma_{i}\, \un_{\zeta_{i}}$ and 
extends $T^{*}$, which thus admits an affine decomposition.

Therefore we can say that if the border rank of $T$ is $r$ then also
$\rank(H_{\Lambda})=r$. Conversely if $\rank(H_{\Lambda})=r$, we can only claim that
the border rank of $T$ is at least $r$. 

We say that  $\Lambda \in  R^{*}$ extends $T^{*}\in
{R^{*}_{\delta_{1},\ldots,\delta_{k}}}$, if
$\Lambda^{*}_{|R_{\delta_{1},\ldots,\delta_{k}}}= T^{*}$. 
The problem of decomposition of $T^{*}$ can then be reformulated as follows:
\par\textbf{Truncated moment problem.} 
{\em Given $T^{*}\in {R^{*}_{\delta_{1},\ldots,\delta_{k}}}$, find the smallest $r$ such that there exists 
  $\Lambda \in  R^{*}$ which extends $T^{*}$  with $H_{\Lambda}$ of rank $r$
  and $I_{\Lambda}$ a radical ideal.}

In the next section, we will describe an algorithm to solve the truncated
moment problem.

\section{Algorithm}
\label{sec:trunc:hankel}\label{sec:5}
In this section, we first describe the algorithm from a geometric point of
view and the algebraic computation it induces). Then we characterize which conditions
$T^{*}$ can be extended to $\Lambda \in \dRR$ with $H_{\Lambda}$ is of
rank $r$. The algorithm is described in \ref{algo:dec}.
It extends the one in \cite{BracCMT09:laa} which applies only for symmetric
tensors. The approach used in \cite{BGI11} for the rank of tensors in 
$\SV_2(S^d(E))$ and in $\SV_3(S^d(E))$ allows to avoid to loop again at step 4: if one 
doesn't get simple roots, then it is possible to use other techniques to
compute the rank. Unfortunately the mathematical knowledge on the
stratification by rank of secant varieties is nowadays not complete, hence
the techniques developped in \cite{BGI11} cannot be used to improve
algorithms for higher border ranks yet.

\begin{algorithm2e}[ht]\caption{\textsc{Decomposition algorithm}}\label{algo:dec}
\KwIn{a tensor $T\in S^{\delta_{1}}(E_{1})\otimes \cdots \otimes
S^{\delta_{k}}(E_{k})$.}
\KwOut{a minimal decomposition of $T$.}
Set $r=0$;
\begin{enumerate}
 \item \label{1} 
  Determine if $T^{*}$ can be extended to $\Lambda \in \dRR$ with $\rank H_{\Lambda}=r$;
 \item \label{2} Find  if there exists $r$ distinct points $P_1, \ldots ,
  P_s\in \SV(S^{\delta_{1}}(E_{1}) \otimes \cdots \otimes
    S^{\delta_{k}}(E_{k}))$ such that $T\in \langle P_1, \ldots , P_s \rangle
      \simeq \Proj^{s-1}$; Equivalently compute the roots of $\ker H_{\Lambda}$
by generalized eigenvector computation \eqref{eq:geiv} and check that the
eigenspaces are simple;
 \item \label{3} If the answer to \ref{2} is YES, then it means that $T\in
  \SV_r^{o}(S^{\delta_{1}}(E_{1})\otimes \cdots \otimes
  S^{\delta_{k}}(E_{k}))\setminus \SV_{r-1}(S^{\delta_{1}}(E_{1})\otimes \cdots \otimes
  S^{\delta_{k}}(E_{k}))$; therefore the rank of $T$
  is actually $r$ and we are done; 
 \item \label{4} If the answer to \ref{2} is NO, then it means that $T\not\in
\SV_r^{o}(S^{\delta_{1}}(E_{1})\otimes \cdots \otimes S^{\delta_{k}}(E_{k}))$ hence its rank is bigger
  than $r$; Repeat this procedure from
  step \ref{2} with $r+1$. 
\end{enumerate}
\end{algorithm2e}

We are going to characterize now under which conditions $T^{*}$ can be
extended to $\Lambda \in \dRR$ with $H_{\Lambda}$ of rank $r$ (step \ref{1}).

We need the following technical property on the bases of $\AA$, that we will consider:
\begin{definition}
Let $B$ be a subset of monomials in $R$. We say that $B$ is
connected to $1$ if $\forall m\in B$ either $m=1$ or there exists $i\in
[1,n]$ and $m'\in B$ such that $m=x_{i}\, m'$.
\end{definition}

Let $B, B' \subset  R_{\delta_{1},\ldots,\delta_{k}}$ be a two sets of monomials connected to $1$.
We consider the formal Hankel matrix 
$$ 
\form{H}^{B,B'}_{\Lambda} = (h_{\alpha+\beta})_{\alpha \in B', \beta\in B},
$$
with $h_{\alpha}=T^{*}(\xb^{\alpha})=c_{\alpha}$ if 
$\xb^{\alpha}\in R_{\delta_{1},\ldots,\delta_{k}}$ and otherwise $h_{\alpha}$ is a variable.
The set of these new variables is denoted $\hb$.

Suppose that $\form{H}^{B,B'}_{\Lambda}$ is invertible in $\kk(\hb)$, then we
define the formal multiplication operators 
$$
\form{M}_{i,l}^{B,B'}(\hb) := (\form{H}^{B,B'}_{\Lambda})^{-1} \form{H}^{B,B'}_{x_{i,l}\star
  \Lambda}
$$
for every variable $x_{i,l} \in R$.\\
We use the following theorems which extend the results of \cite{ML08} to the cases of
distinct sets of monomials indexing the rows and columns of the  Hankel operators.
They characterizes the cases where $\kk[\xb] = B \oplus I_{\Lambda}$:
\begin{theorem}\label{th:commute}
Let $B=\{\xb^{\beta_{1}},\ldots,\xb^{\beta_{r}}\}$ and
$B'=\{\xb^{\beta'_{1}}$, $\ldots$, $\xb^{\beta'_{r}}\}$ be two sets of monomials of
in $R_{\delta_{1},\ldots,\delta_{k}}$, connected to $1$ and let $\Lambda$ be a linear form that belongs to
$(\<B'\cdot B^{+}\>_{\delta_{1},\ldots,\delta_{k}})^*$. Let $\Lambda(\hb)$ be the linear form of
$\dual{\<B'\cdot B^{+}\>}$ defined by
$\Lambda(\hb)(\xb^{\alpha})=\Lambda(\xb^{\alpha})$ if $\xb^{\alpha}\in R_{\delta_{1},\ldots,\delta_{k}}$ and
$h_{\alpha} \in \kk$ otherwise. Then, $\Lambda(\hb)$ admits an extension
$\tilde{\Lambda} \in \dRR$ such that $H_{\tilde{\Lambda}}$ is of rank $r$
with $B$ and $B'$ basis of $A_{\tilde{\Lambda}}$ iff  
\begin{equation}
  \mathcal{M}_{i,l}^{B}(\hb) \circ \mathcal{M}_{j,q}^{B}(\hb)
  -\mathcal{M}_{j,q}^{B}(\hb) \circ \mathcal{M}_{i,l}^{B}(\hb)=0
\end{equation}
$(0 \le l,q\le k, 1 \le i\le n_{l}, 1 \le j \le n_{q})$
and $\det(\mathcal{H}^{B',B}_{\Lambda(\hb)}) \neq 0$. Moreover, such a $\tilde{\Lambda}$ is unique.
\end{theorem}
\begin{proof} 
If there exists $\tilde{\Lambda} \in \dRR$ which extends
$\Lambda(\hb)$, with $H_{\tilde{\Lambda}}$ of rank $r$  and $B$ and $B'$ basis of $A_{\tilde{\Lambda}}$ then $\mathcal{H}^{B',B}_{\Lambda(\hb)}$ is invertible and the tables of multiplications by the variables $x_{i,l}$:
$$
\mathcal{M}_{i,l}:=(\mathcal{H}^{B',B}_{\Lambda(\hb)})^{-1} \mathcal{H}^{B',B}_{x_{i,l}\star \Lambda(\hb)}
$$
(Proposition \ref{prop:Hl-mx}) commute.

Conversely suppose that these matrices commute and consider them 
as linear operators on $\langle B \rangle$.
Then by \cite{mourrain-aaecc-1999}, we have such a decomposition
$R = \<B\> \oplus I$ where $I$ is an ideal of $R$. As a matter of fact, using
commutation relation and the fact that $B$ is connected to 1, one can easily
prove that the following morphism: 
$$
\pi: R \longrightarrow \langle B \rangle
$$
$$
p \rightarrow p(\mathcal{M})(1)
$$
is a projection on $\langle B \rangle$ whose kernel is an ideal $I$ of $R$ (note that for any $p \in \langle B \rangle$, $p(M)(1) = p$). 

We define $\tilde{\Lambda} \in \dRR$ as follows:  
$\forall p\in \RR, \tilde{\Lambda}(p)=\Lambda(p(\form{M})(1))$ where
$p(\mathcal{M})$ is the operator obtained by substitution of the variables
$x_{i,l}$ by the commuting operators $\mathcal{M}_{i,l}$.
Notice that $p(\mathcal{M})$ is also the operator of multiplication by $p$
modulo $I$. 
 
Let us prove by induction on the degree of $b'\in B'$ that for all $b\in B:$ 
\begin{equation}\label{eq:comm1}
\Lambda(\mathbf{h})(b' \, b)=\Lambda(\mathbf{h})(b'(\mathcal{M})(b))
\end{equation}
 and thus by linearity that
 \begin{equation}\label{eq:comm2}
\Lambda(\mathbf{h})(b'\,p)=\Lambda(\mathbf{h})(b'(\mathcal{M})(p))
 \end{equation}
for all $p \in \langle B\rangle$.

The property is obviously true for $b'=1$. Suppose now that $b' \in B'$ is a monomial of degree strictly greater than zero. As $B'$ is connected to $1$, one has $b' = x_{j,q}\, b''$ for some
variable $x_{j,q} \in R$ and some element $b''\in B'$ of degree smaller than the degree of $b'$. 
By construction of the operators of multiplication $(\mathcal{M}_{i,l})$, we have 
$$
\Lambda(\mathbf{h})(b'\,b)= \Lambda(\mathbf{h})(b''\,x_{j,q}\,b) = \Lambda(\mathbf{h})(b''\,\mathcal{M}_{j,q}(b)).
$$
Finally we have that $\Lambda(\mathbf{h})(b'\,b)=\Lambda(\mathbf{h})(b'(\mathcal{M})(b))$ and \eqref{eq:comm1} is proved.\\
Let us deduce now that $\tilde{\Lambda}$ extends $\Lambda(\mathbf{h})$ i.e that for all $b^+\in B^+$ ans $b' \in B'$ we have: 
$$
\tilde{\Lambda}(b'\,b^+):=\Lambda(\mathbf{h})((b'\,b^+)(\mathcal{M})(1)) = \Lambda(\mathbf{h})(b'\,b^+).
$$
Indeed, from \eqref{eq:comm2} we have: 
$$
\Lambda(\mathbf{h})((b'\,b^+)(\mathcal{M})(1))=\Lambda(\mathbf{h})((b'(\mathcal{M})\,b^+(\mathcal{M})(1)) = \Lambda(\mathbf{h})(b'\,[b^+(\mathcal{M})(1)])
$$
as $b^+(\mathcal{M})(1)$ belongs to $\langle B \rangle$. Then, by definition of multiplication operators $(\mathcal{M}_{i,l})$ we have
$$
\Lambda(\mathbf{h})(b'\,[b^+(\mathcal{M})(1)]) = \Lambda(\mathbf{h})(b'\,b^+).
$$
Thus, we have 
\begin{equation}\label{eq:comm3}
\tilde{\Lambda}(b'\,b^+) = \Lambda(\mathbf{h})(b'\,b^+)
\end{equation}
for all $b^+\in B^+$ ans $b' \in B'$ (i.e $\tilde{\Lambda}$ extends $\Lambda(\mathbf{h})$).\\

We eventually need to prove that $I_{\tilde{\Lambda}} = I:=\text{Ker}(\pi)$. By the definition of $\tilde{\Lambda}$ we obviously have that $I \subset I_{\tilde{\Lambda}}$. Let us prove that $I_{\tilde{\Lambda}} \subset I$: assume $p$ belongs to $I_{\tilde{\Lambda}}$, then from \eqref{eq:comm3}
$$
\tilde{\Lambda}(b'\,p(\mathcal{M})(1)) = \Lambda(\mathbf{h})(b'\,p(\mathcal{M})(1)) = 0
$$
for all $b' \in B'$. As $p(\mathcal{M})(1) \in \langle B \rangle$ and $\det(\mathcal{H}^{B',B}_{\Lambda})(\hb) \neq 0$, we deduce that $p(\mathcal{M})(1) = 0$ and that $p$ belongs to $I$. Thus we have $I_{\tilde{\Lambda}} \subset I$.\\

Eventually, $\tilde{\Lambda}$ extends $\Lambda(\mathbf{h})$ with $I_{\tilde{\Lambda}} = I:=\text{Ker}(\pi)$ and $\mathcal{A}_{\tilde{\Lambda}}$ equal to $R/I \simeq \langle B \rangle$ which is a zero dimensional algebra of multiplicity $r$ with basis $B$.

If there exists another $\Lambda' \in \dRR$ which extends $\Lambda(\hb) \in
\dual{\<B'\cdot B^{+}\>}$ with $\rank H_{\Lambda'}=r$, by proposition
\ref{prop:basis:ideal}, $\ker H_{\Lambda'}$ is generated by
$\ker H_{\Lambda'}^{B',B^{+}}$ and thus coincides with $\ker H_{\tilde{\Lambda}}$. As $\Lambda'$ coincides with $\tilde{\Lambda}$ on $B$,
the two elements of $\dRR$ must be equal.  This ends the proof of the
theorem.
\end{proof}

The degree of these commutation relations is at most $2$ in the coefficients
of the multiplications matrices $\form{M}_{i,l}$. A direct computation yields the following, for 
$m\in B$:
\begin{itemize}
\item If $x_{i,l}, m \in B,$ and $x_{j,q}\,m \in B$ then $(\form{M}_{i,l}^{B} \circ \form{M}_{j,q}^{B}
  -\form{M}_{j,q}^{B} \circ \form{M}_{i,l}^{B})(m)\equiv 0$ in $\kk(\hb)$.
\item If $x_{i,l} m \in B$, $x_{j,q}\,m \not\in B$ then 
  $(\form{M}_{i,l}^{B} \circ \form{M}_{j,q}^{B} - \form{M}_{j,q}^{B} \circ
  \form{M}_{i,l}^{B})(m)$ is of degree $1$ in the coefficients of
  $\form{M}_{i,l},\form{M}_{j,q}$.
\item If $x_{i,l} m \not\in B$, $x_{j,q}\,m \not\in B$ then 
  $(\form{M}_{i,l}^{B} \circ \form{M}_{j,q}^{B} - \form{M}_{j,q}^{B}
  \circ \form{M}_{i,l}^{B})(m)$ is of degree $2$ in the coefficients of
  $\form{M}_{i,l},\form{M}_{j,q}$.
\end{itemize}

We are going to give an equivalent characterization of the extension 
property, based on rank conditions:
\begin{theorem}\label{thm:rank}
Let $B=\{\xb^{\beta_{1}},\ldots,\xb^{\beta_{r}}\}$ and
$B'=\{\xb^{\beta_{1}'}$, $\ldots$, $\xb^{\beta_{r}'}\}$ be two sets of monomials in $R_{\delta_{1},\ldots,\delta_{k}}$, connected to $1$. Let $\Lambda$ be a linear form in $(\langle B'^+*B^+ \rangle_{\delta_{1},\ldots,\delta_{k}})^*$ and  $\Lambda(\hb)$ be the linear form of
$\dual{\<B'^+\cdot B^{+}\>}$ defined by
$\Lambda(\hb)(\xb^{\alpha})=\Lambda(\xb^{\alpha})$ if $\xb^{\alpha}\in R_{\delta_{1},\ldots,\delta_{k}}$ and
$h_{\alpha} \in \kk$ otherwise. Then, $\Lambda(\hb)$ admits an extension
$\tilde{\Lambda} \in \dRR$ such that $H_{\tilde{\Lambda}}$ is of rank $r$
with $B$ and $B'$ basis of $A_{\tilde{\Lambda}}$ iff all
$(r+1)\times(r+1)$-minors of $\mathcal{H}^{B'^+,B^{+}}_{\Lambda(\hb)}$ vanish and
$\det(\mathcal{H}^{B',B}_{\Lambda(\hb)}) \neq 0$.
\end{theorem}
\begin{proof}
First, if such a $\tilde{\Lambda}$ exists then $\mathcal{H}^{B',B}_{\Lambda(\hb)}$ is invertible and $\mathcal{H}_{\tilde{\Lambda}}$ is of rank $r$. Thus all the $(r+1)\times(r+1)$-minors of $\mathcal{H}^{B'^+,B^{+}}_{\Lambda(\hb)}$ are equal to zero.\\

Reciprocally, assume all $(r+1)\times(r+1)$-minors of $\mathcal{H}^{B'^+,B^{+}}_{\Lambda(\hb)}$ vanish and $\det(\mathcal{H}^{B',B}_{\Lambda(\hb)}) \neq 0$ then one can consider the same operators:
$$
\form{M}_{i,l}^{B,B'}(\hb) := (\form{H}^{B,B'}_{\Lambda(\hb)})^{-1} \form{H}^{B,B'}_{x_{i,l}\star
  \Lambda(\hb)}.
$$
By definition of these operators one has that
\begin{equation}\label{eq:det1}
\Lambda(\hb)(x_{i,l}b\,b') = \Lambda(\hb)(\mathcal{M}_{i,l}(b)\,b')
\end{equation}
for all $b \in B$ and $b' \in B'$. As the rank of $\mathcal{H}^{B'^+,B^{+}}_{\Lambda(\hb)}$ is equal to the rank of $\mathcal{H}^{B',B}_{\Lambda(\hb)}$ we easily deduce that \eqref{eq:det1} is also true for $b' \in B'^+$. Thus we have 
$$
\Lambda(\hb)(x_{j,q}x_{i,l}b\,b') = \Lambda(\hb)(x_{i,l}b\,b'x_{j,q}) = \Lambda(\hb)(\mathcal{M}_{i,l}(b)\,b'x_{j,q}) = 
$$
$$
= \Lambda(\hb)(x_{j,q}\,\mathcal{M}_{i,l}(b)\,b') = \Lambda(\hb)([\mathcal{M}_{j,q}\,\mathcal{M}_{i,l}](b)\,b')
$$
for all $b\in B$, $b'$ in $B'$ and $x_{i,l},x_{j,q} \in R$. Then,
$$
\Lambda(\hb)([\mathcal{M}_{j,q}\,\mathcal{M}_{i,l}](b)\,b') = \Lambda(\hb)([\mathcal{M}_{i,l}\,\mathcal{M}_{j,q}](b)\,b') = \Lambda(\hb)(x_{j,q}x_{i,l}b\,b')
$$
for all $b\in B$ and $b' \in B'$. As $\mathcal{H}^{B',B}_{\Lambda(\hb)}$ is invertible we deduce 
$$
[\mathcal{M}_{j,q}\,\mathcal{M}_{i,l}](b) = [\mathcal{M}_{i,l}\,\mathcal{M}_{j,q}](b)
$$
for all $b\in B$. Thus
$$
\mathcal{M}_{j,q}\,\mathcal{M}_{i,l} = \mathcal{M}_{i,l}\,\mathcal{M}_{j,q}.
$$
Finally, we conclude the proof by using Theorem \ref{th:commute}\,
\end{proof}
\begin{proposition}\label{prop:dec}
Let $B=\{\xb^{\beta_{1}},\ldots,\xb^{\beta_{r}}\}$ and
$B'=\{\xb^{\beta'_{1}}$, $\ldots$, $\xb^{\beta'_{r}}\}$ be two sets of monomials of
in $R_{\delta_{1},\ldots,\delta_{k}}$, connected to $1$ and $\Lambda$ be a linear form on $\langle B^{'+} * B^{+} \rangle$. Then, $\Lambda$ admits an extension $\tilde{\Lambda}$ in $\dRR$ such that
$H_{\tilde{\Lambda}}$ is of rank $r$ with $B$ (resp. B') a basis of $\mathcal{A}_{\tilde{\Lambda}}$ iff 
\begin{equation}\label{dec:H}
\mathbb{H}^{+}
  =
  \left(
    \begin{array}{cc}
      \mathbb{H} & \mathbb{G}' \\
      \mathbb{G}^{t} & \mathbb{J}  
    \end{array}
  \right),
\end{equation}
with 
$\mat{H}^{+}=\mathbb{H}^{B^{'+},B^{+}}_{\Lambda}$
$\mathbb{H}= \mathbb{H}^{B',B}_{\Lambda}$ and 
\begin{equation}\label{eq:W}
\mathbb{G}=\mathbb{H}^{t}\, \mathbb{W},
\mathbb{G}'=\mathbb{H}\, \mathbb{W}',
\mathbb{J}=\mathbb{W}^{t}\,\mathbb{H}\,\mathbb{W}'. 
\end{equation} 
for some matrix $\mathbb{W}\in \kk^{B \times \partial B'}$, $\mathbb{W}'\in \kk^{B' \times \partial B}$.
\end{proposition}

\begin{proof} 
According to Theorem \ref{thm:rank}, $\Lambda \in \langle B^{'+} * B^+ \rangle^*$ admits a (unique)
extension $\tilde{\Lambda} \in \dRR$ such that 
$H_{\tilde{\Lambda}}$ is of rank $r$ with $B$ (resp. B') a basis of $\mathcal{A}_{\tilde{\Lambda}}$, iff 
$H^{B^{'+},B^{+}}_{\tilde{\Lambda}}=H^{B^{'+},B^{+}}_{\Lambda}$ is of rank $r$.
Let us decompose $H^{+}_{\Lambda}$ as \eqref{dec:H} with 
$\mat{H}^{+}=\mat{H}^{B^{'+},B^{+}}_{\Lambda}$,
$\mat{H}=\mat{H}^{B',B}_{\Lambda}$. 

If we have $\mat{G}=\mat{H}^{t}\, \mat{W}, \mathbb{G}'=\mathbb{H}\, \mathbb{W}',
  \mat{J}=\mat{W}^{t}\,\mat{H}\,\mat{W}'$, then 
$$
\mat{H}^{+} = \left(
    \begin{array}{cc}
      \mat{H} & \mat{H}\, \mat{W}' \\
      \mat{W}^{t} \mat{H} & \mat{W}^{t}\,\mat{H}\,\mat{W}'
    \end{array}
  \right)
$$
is clearly of rank  $\le \rank \mat{H}$.

Conversely, suppose that $\rank \mat{H}^{+} = \rank \mat{H}$. 
This implies that the image of $\mat{G}'$ is in the image of $\mat{H}$. Thus,
there exists  $\mat{W}'\in \kk^{B' \times \partial B}$ such that
$\mat{G}'=\mat{H}\, \mat{W}'$. Similarly, there exists  $\mat{W}\in \kk^{B
  \times \partial B'}$ such that $\mat{G}=\mat{H}^{t}\, \mat{W}$.
Thus, the kernel of $( \mat{H}\  \mat{G}')$ (resp. $( \mat{H}^{t}\ \mat{G})$) contains 
$\left(
\begin{array}{c}
\mat{W}' \\
-\mat{I}
\end{array}\right)$ 
(resp. $\left(
\begin{array}{c}
\mat{W} \\
-\mat{I}
\end{array}\right)$).
As $\rank \mat{H}= \rank \mat{H}^{+}=r$, the kernel of $( \mat{H}\ \mat{G}')$ (resp. $( \mat{H}^{t}\ \mat{G})$) is the
kernel of $\mat{H}^{+}$ (resp $\mat{H}^{+t}$).
 Thus we have $\mat{J}= \mat{G}^{t}\, \mat{W}'=  \mat{W}^{t}\,\mat{H}\,\mat{W}'$. 
\end{proof}
Notice that if $\mat{H}$ is invertible, $\mat{W}'$, $\mat{W}$ are uniquely determined.

Introducing new variables $\mathbf{w}$, $\mathbf{w}'$ for the coefficients of the matrices $\mat{W},
\mat{W}$, solving the linear system
$\mathbb{G}=\mathbb{H}^{t}\, \mathbb{W},\mathbb{G}'=\mathbb{H}\,
\mathbb{W}'$, and reporting the solutions in the equation
$\mat{J}=\mat{W}^{t}\,\mat{H}\,\mat{W}'$, we obtain 
a new set of equations, bilinear in $\mathbf{w}$, $\mathbf{w}'$, which
characterize the existence of an extension $\Lambda$ on $R^{*}$.

\ifx \@full \@empty \releax \else
This leads to the following system in the variables $\mathbf{h}$ and the coefficients
$\mathbf{w}$ of matrix $\mathbb{W}$. It characterizes the linear forms $\Lambda \in
R_{\delta_{1},...,\delta_{k}}^*$ that admit an extension $\tilde{\Lambda} \in \dRR$ such that
$H_{\tilde{\Lambda}}$ is of rank $r$ with $B$ a basis of $\mathcal{A}_{\tilde{\Lambda}}$.
\begin{equation} \label{syst:GH}
\mathcal{H}_{\Lambda}^{B,\partial B}(\mathbf{h}) - \mathcal{H}_{\Lambda}^{B}(\mathbf{h})\,
\mathbb{W}(\mathbf{w}) =0, \ \ 
\mathcal{H}_{\Lambda}^{\partial B, \partial B}(\mathbf{h}) - 
\mathbb{W}^{t}(\mathbf{w})\,\mathcal{H}_{\Lambda}^{B}(\mathbf{h})\,\mathbb{W}(\mathbf{w}) =0
\end{equation}
with $\det(\mathcal{H}_{\Lambda}^{B}(\mathbf{h}))\neq 0$.

The matrix $\mathcal{H}^{B^{+}}_{\Lambda}$ is a quasi-Hankel matrix
\cite{mp-jcomplexity-2000}, whose structure is imposed by equality (linear)
constraints on its entries.
If $\mathbb{H}$ is known (i.e. $B\times B \subset R_{\delta_{1},...,\delta_{k}}$, the number of
independent parameters in $\mathcal{H}_{\Lambda}^{B,B^{+}}(\mathbf{h})$ or in
$\mathbb{W}$ is the number  of monomials in $B \times \partial B - R_{\delta_{1},...,\delta_{k}}$.   
By Proposition \ref{prop:dec}, the rank condition is equivalent to the quadratic
relations $\mathbb{J}- \mathbb{W}^{t} \mathbb{H}^{t}\, \mathbb{W}=0$ in these unknowns. 

If $\mathbb{H}$ is not completely known, the number of parameters in $\mathbb{H}$ 
is the number of monomials in $B\times B - R_{\delta_{1},...,\delta_{k}}$. The number of independent parameters in
$\mathcal{H}_{\Lambda}^{B,\partial B}(\mathbf{h})$ or in $\mathbb{W}$ is then $B\times
\partial B - R_{\delta_{1},...,\delta_{k}}$. 

The system \eqref{syst:GH} is composed of linear equations deduced from 
 quasi-Hankel structure, quadratic relations  for the entries in $B\times
\partial B$
and cubic relations  for the entries in $B\times \partial B$ in the unknown
parameters $\mathbf{h}$ and $\mathbf{w}$.

We are going to use explicitly these characterizations in the new algorithm
we propose for minimal tensor decomposition.
%
%
%

\section{Examples and applications}
There exist numerous fields in which decomposing a tensor into a sum of
rank-one terms is useful. These fields range from arithmetic complexity
\cite{BurgCS97} to chemistry \cite{SmilBG04}. One nice application is worth
to be emphasized, namely wireless transmissions \cite{SidiGB00:ieeesp}: one
or several signals are wished to be extracted form noisy measurements,
received on an array of sensors and disturbed by interferences. The approach
is deterministic, which makes the difference compared to approaches based on
data statistics \cite{ComoJ10}.  The array of sensors is composed of $J$
subarrays, each containing $I$ sensors. Subarrays do not need to be disjoint,
but must be deduced from each other by a translation in space. If the
transmission is narrow band and in the far field, then the measurements at
time sample $t$ recorded on sensor $i$ of subarray $j$ take the form:

\centerline{$T(i,j,t) = \sum_{p=1}^r A_{ip} B_{jp} C_{tp}$}

\vspace{1mm}\noindent{}if $r$ waves impinge on the array. Matrices $A$ and $B$ characterize the geometry of the array (subarray and translations), whereas matrix $C$ contains the signals received on the array. 
An example with $(I,J)=(4,4)$ is given in Figure \ref{array4-fig}.
Computing the decomposition of tensor $T$ allows to extract signals of interest as well as interferences, all included in matrix $C$. Radiating sources can also be localized with the help of matrix $A$ if the exact location of sensors of a subarray are known.
Note that this framework applies in  radar, sonar or telecommunications.

\begin{figure}[htb]\begin{center}
\includegraphics{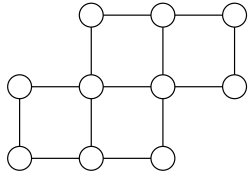}\caption{Array of 10 sensors decomposed into 4 subarrays of 4 sensors each.}\label{array4-fig}
\end{center}\end{figure}
\subsection{Best approximation of lower multilinear rank}
By considering a $k$th order tensor as a linear map from one linear space onto the tensor product of the others, one can define the $i$th mode rank, which is nothing else but the rank of that linear operator. Since there are $k$ distinct possibilities to build such a linear operator, one defines a $k$-uplet of ranks $(r_1,\dots r_k)$, called the \emph{multilinear rank} of the $k$th order tensor.
It is known that tensor rank is bounded below by all mode ranks $r_i$:
\begin{equation}\label{lowerBound-eq}
r \ge r_i , \forall 1\le i \le k
\end{equation} 
This inequality gives us  an easily accessible lower bound. 
Let's turn now to an upper bound.
\begin{proposition}\cite{AtkiS79:laa}
The rank of a tensor of order 3 and dimensions $n_1\times n_2\times n_3$, with $n_1\le n_2$, is bounded by
\begin{equation}\label{upperBound-eq}
r \le n_1 + n_2 \lfloor n_3/2\rfloor 
\end{equation}
\end{proposition}
This bound on maximal rank has not been proved to be always reached, and it is likely to be quite loose for large values of $n_i$. Nevertheless, it is sufficient for our reasoning.

\smallskip

There are two issues to address.
First, the algorithm we have proposed is not usable in large dimensions (e.g. significantly larger than 10). The idea is then to reduce dimensions $n_i$ down to $r_i$ before executing the algorithm, if necessary. 
Second, another problem in practice is the presence of measurement errors or modeling inaccuracies, which increase the tensor rank to its generic value. We do not know how to reduce tensor rank back to its exact value. The practical solution is then to compute the best approximate of lower multilinear rank $(r_1,\dots r_k)$, as explained in \cite{ComoLA09:jchemo}. This best approximate always exists, and inequality (\ref{upperBound-eq}) shows that reducing dimensions will indirectly reduce tensor rank.
To compute it, it suffices to minimize $||T - (U^{(1)},U^{(2)},U^{(3)})\cdot C||$ with respect to the three matrices $U^{(i)}$, each of size $n_i\times r_i$, under the constraint $U^{(1)\sf H}U^{(1)}=I$. If properly initialized by a truncated HOSVD, a few iterations of any iterative algorithm will do it \cite{DelaDV00:simax}. The tensor of reduced dimensions is then given by $C=(U^{(1)\sf H},U^{(2)\sf H},U^{(3)\sf H})\cdot T$.

\subsection{Number of solutions}
In the above mentioned applications, it is necessary to have either a unique solution, or a finite set of solutions from which the most realistic one can be picked up. For this reason, it is convenient to make sure that the tensor rank is not too large, as pointed out by the following propositions.

\begin{proposition}\cite{ChiaC01:tams}
A generic symmetric tensor of order $k\ge3$ and rank $r$ admits a finite number of decompositions into a sum of rank one terms if $r<r_E(k,n)$, where:
\begin{equation}
r_E^{}(k,n) = \left\lceil\frac{\binom{n+k-1}{d}}{n}\right\rceil
\end{equation}
Rank $r_E^{}$ is usually referred to as the expected rank of order $k$ and dimension $n$. 
\end{proposition}
Note that this result is true for generic tensors of rank $r$, which means that there exists a set of exceptions, of null measure.

This proposition has not yet been entirely extended to unconstrained tensors, which we are interested in. However, some partial results are available in the literature \cite{AboOP09:tams}, and the following conjecture is generally admitted
\begin{conjecture}
A generic tensor of order $k\ge3$ and rank $r$ admits a finite number of decompositions into a sum of rank one terms if $r<r_E(k,\mathbf{n})$, where:
\begin{equation}
r_E^{}(k,\mathbf{n}) = \left\lceil\frac{\prod_{i=1}^k n_1}{1-k+\sum_{i=1}^k n_i}\right\rceil
\end{equation}
where $n_i$ denote the dimensions, $1\le i\le k$ 
\end{conjecture}

On the other hand, a sufficient condition for uniqueness has been proposed by Kruskal \cite{Krus77:laa}, but the bound is more restrictive:
\begin{proposition}\cite{Krus77:laa}
A tensor of order $k\ge3$ and rank $r$ admits a finite number of decompositions into a sum of rank one terms if:
\begin{equation}
r \le \frac{\sum_{i=1}^k \kappa_i}{2}-1
\end{equation}
where $\kappa_i$ denote the so-called Kruskal's ranks of loading matrices, which generically equal the dimensions $n_i$ if the rank $r$  is larger than the latter.
\end{proposition}

\subsection{Computer results}
If we consider a $4\times4\times7$ unconstrained tensor, it has an expected rank equal to 9, whereas Kruskal's bound generically equals 6.
So it is interesting to consider a tensor with such dimensions but with rank $6<r<9$. In such conditions, we expect that there are almost surely a finite number of solutions. 
This tensor would correspond to measurements received on the array depicted in Figure \ref{array4-fig}, if $7$ time samples are recorded.  In \cite{Chiantini:2011fk} L. Chiantini and G. Ottaviani claim that a computer check shows that for a generic $4\times4\times7$ tensor of rank $7$, uniqueness of the decomposition holds.

\smallskip

Our computer results have been obtained with $4\times4\times7$ tensors of rank 7 randomly drawn according to a continuous probability distribution. 

First we consider a $4\times4\times4$ tensor whose affine representation is given by:\\
$
T:=4+7\,a_1+8\,a_2+9\,a_3+5\,b_1-2\,b_2+11\,b_3+6\,c_1+8\,c_2+6\,c_3+21\,a_1\,b_1+28\,a_2\,b_1+11\,a_3\,b_1-14\,a_1\,b_2-21\,a_2\,b_2-10\,a_3\,b_2+48\,a_1\,b_3+65\,a_2\,b_3+28\,a_3\,b_3+26\,a_1\,c_1+35\,a_2\,c_1+14\,a_3\,c_1+18\,b_1\,c_1-10\,b_2\,c_1+40\,b_3\,c_1+36\,a_1\,c_2+48\,a_2\,c_2+18\,a_3\,c_2+26\,b_1\,c_2-9\,b_2\,c_2+55\,b_3\,c_2+38\,a_1\,c_3+53\,a_2\,c_3+14\,a_3\,c_3+26\,b_1\,c_3-16\,b_2\,c_3+58\,b_3\,c_3+68\,a_1\,b_1\,c_1+91\,a_2\,b_1\,c_1+48\,a_3\,b_1\,c_1-72\,a_1\,b_2\,c_1-105\,a_2\,b_2\,c_1-36\,a_3\,b_2\,c_1+172\,a_1\,b_3\,c_1+235\,a_2\,b_3\,c_1+112\,a_3\,b_3\,c_1+90\,a_1\,b_1\,c_2+118\,a_2\,b_1\,c_2+68\,a_3\,b_1\,c_2-85\,a_1\,b_2\,c_2-127\,a_2\,b_2\,c_2-37\,a_3\,b_2\,c_2+223\,a_1\,b_3\,c_2+301\,a_2\,b_3\,c_2+151\,a_3\,b_3\,c_2+96\,a_1\,b_1\,c_3+129\,a_2\,b_1\,c_3+72\,a_3\,b_1\,c_3-114\,a_1\,b_2\,c_3-165\,a_2\,b_2\,c_3-54\,a_3\,b_2\,c_3+250\,a_1\,b_3\,c_3+343\,a_2\,b_3\,c_3+166\,a_3\,b_3\,c_3.$\\

If we consider $B':=(1,b_1,b_2,b_3)$ and $B:= (1,a_1,a_2,a_3)$, the corresponding matrix $\mathbb{H}_{\Lambda}^{B',B}$ is equal to
$$
\mathbb{H}_{\Lambda}^{B',B} = \left(
    \begin{array}{cccc}
      4 & 7 & 8 & 9\\
      5 & 21 & 28 & 11\\
      -2 & -14 & -21 & -10\\
      11 & 48 & 65 & 28\\
    \end{array}
  \right)
$$
and is invertible. Moreover, the transposed operators of multiplication by the variables $c_1,c_2,c_3$ are known:
$$
^t \mathbb{M}_{c_1}^{B}= \left(
    \begin{array}{cccc}
      0 & 11/6 & -2/3 & -1/6\\
      -2 & -41/6 & 20/3 & 19/6\\
      -2 & -85/6 & 37/3 & 29/6\\
      -2 & 5/2 & 0 & 1/2\\
    \end{array}
  \right)
$$

$$
^t \mathbb{M}_{c_2}^{B}= \left(
    \begin{array}{cccc}
      -2 & 23/3 & -13/3 & -1/3\\
      -6 & 1/3 & 7/3 & 13/3\\
      -6 & -28/3 & 29/3 & 20/3\\
      -6 & 14 & -7 & 0\\
    \end{array}
  \right)
$$

$$
^t \mathbb{M}_{c_3}^{B}= \left(
    \begin{array}{cccc}
      0 & 3/2 & 0 & -1/2\\
      -2 & -33/2 & 14 & 11/2\\
      -2 & -57/2 & 23 & 17/2\\
      -2 & 3/2 & 2 & -1/2\\
    \end{array}
  \right)
$$
whose eigenvalues are respectively $(-1,2,4,1)$, $(-2,4,5,1)$ and $(-3,2,6,1)$. The corresponding common eigenvectors are:
$$
v_1= \left(
    \begin{array}{c}
      1 \\
      -1\\
      -2\\
      3\\
    \end{array}
  \right),
v_2= \left(
    \begin{array}{c}
      1 \\
      2\\
      2\\
      2\\
    \end{array}
  \right),
v_3= \left(
    \begin{array}{c}
      1\\
      5\\
      7\\
      3\\
    \end{array}
  \right),
v_4= \left(
    \begin{array}{c}
      1 \\
      1\\
      1\\
      1\\
    \end{array}
  \right).
$$
We deduce that the coordinates $(a_1,a_2,a_3,b_1,b_2,b_3,c_1,c_2,c_3)$ of
the 4 points of evaluation are:
$$
\zeta_1:= \left(
    \begin{array}{c}
      -1 \\
      -2\\
      3\\
      ?\\
      ?\\
      ?\\
      -1\\
      -2\\
      -3\\
    \end{array}
  \right),
 \zeta_2:= \left(
    \begin{array}{c}
      2 \\
      2\\
      2\\
      ?\\
      ?\\
      ?\\
      2\\
      4\\
      2\\
    \end{array}
  \right),
\zeta_3:= \left(
    \begin{array}{c}
      5 \\
      7\\
      3\\
      ?\\
      ?\\
      ?\\
      4\\
      5\\
      6\\
    \end{array}
  \right),
\zeta_4:= \left(
    \begin{array}{c}
      1 \\
      1\\
      1\\
      ?\\
      ?\\
      ?\\
      1\\
      1\\
      1\\
    \end{array}
  \right),
$$
Then, computing the same way the operators of multiplication $^t \mathbb{M}_{c_1}^{B'},^t \mathbb{M}_{c_2}^{B'},^t \mathbb{M}_{c_3}^{B'}$ and their common eigenvectors, we deduce:
$$
\zeta_1= \left(
    \begin{array}{c}
      -1 \\
      -2\\
      3\\
      -1\\
      -1\\
      -1\\
      -1\\
      -2\\
      -3\\
    \end{array}
  \right),
 \zeta_2:= \left(
    \begin{array}{c}
      2 \\
      2\\
      2\\
      2\\
      2\\
      3\\
      2\\
      4\\
      2\\
    \end{array}
  \right),
\zeta_3= \left(
    \begin{array}{c}
      5 \\
      7\\
      3\\
      3\\
      -4\\
      8\\
      4\\
      5\\
      6\\
    \end{array}
  \right),
\zeta_4= \left(
    \begin{array}{c}
      1 \\
      1\\
      1\\
      1\\
      1\\
      1\\
      1\\
      1\\
      1\\
    \end{array}
  \right).
$$
Finally, we have to solve the following linear system in $(\gamma_1,\gamma_2,\gamma_3,\gamma_4)$:\\
\ \\

$T = \gamma_1 \,(1+a_1+a_2+a_3)\,(1+b_1+b_2+b_3)\,(1+c_1+c_2+c_3)\\
+\gamma_2\,(1-a_1-2\,a_2+3\,a_3)\,(1-b_1-b_2-b_3)\,(1-c_1-2\,c_2-3\,c_3)\\
+\gamma_3\,(1+2\,a_1+2\,a_2+2\,a_3)\,(1+2\,b_1+2\,b_2+3\,b_3)\,(1+2\,c_1+4\,c_2+2\,c_3) \\
+\gamma_4\,(1+5\,a_1+7\,a_2+3\,a_3)\,(1+3\,b_1-4\,b_2+8\,b_3)\,(1+4\,c_1+5\,c_2+6\,c_3)$,
\ \\
We get $\gamma_1=\gamma_2=\gamma_3=\gamma_4=1$.\\
We consider such an example with $6$ time samples, that is an
element of $\RR^{4}\otimes\RR^{4}\otimes \RR^{6}$:
$T:=${ $1046\,a_1\,b_1\,c_1+959\,a_1\,b_1\,c_2+660\,a_1\,b_1\,c_3+866\,a_1\,b_1\,c_4+952\,a_1\,b_1\,c_5-1318\,a_1\,b_2\,c_1-1222\,a_1\,b_2\,c_2-906\,a_1\,b_2\,c_3-1165\,a_1\,b_2\,c_4-1184\,a_1\,b_2\,c_5-153\,a_1\,b_3\,c_1+52\,a_1\,b_3\,c_2+353\,a_1\,b_3\,c_3+354\,a_1\,b_3\,c_4+585\,a_1\,b_3\,c_5+852\,a_2\,b_1\,c_1+833\,a_2\,b_1\,c_2+718\,a_2\,b_1\,c_3+903\,a_2\,b_1\,c_4+828\,a_2\,b_1\,c_5-1068\,a_2\,b_2\,c_1-1060\,a_2\,b_2\,c_2-992\,a_2\,b_2\,c_3-1224\,a_2\,b_2\,c_4-1026\,a_2\,b_2\,c_5+256\,a_2\,b_3\,c_1+468\,a_2\,b_3\,c_2+668\,a_2\,b_3\,c_3+748\,a_2\,b_3\,c_4+1198\,a_2\,b_3\,c_5-614\,a_3\,b_1\,c_1-495\,a_3\,b_1\,c_2-276\,a_3\,b_1\,c_3-392\,a_3\,b_1\,c_4-168\,a_3\,b_1\,c_5+664\,a_3\,b_2\,c_1+525\,a_3\,b_2\,c_2+336\,a_3\,b_2\,c_3+472\,a_3\,b_2\,c_4+63\,a_3\,b_2\,c_5+713\,a_3\,b_3\,c_1+737\,a_3\,b_3\,c_2+791\,a_3\,b_3\,c_3+965\,a_3\,b_3\,c_4+674\,a_3\,b_3\,c_5-95\,a_1\,b_1+88\,a_1\,b_2+193\,a_1\,b_3+320\,a_1\,c_1+285\,a_1\,c_2+134\,a_1\,c_3+188\,a_1\,c_4+382\,a_1\,c_5-29\,a_2\,b_1-2\,a_2\,b_2+198\,a_2\,b_3+292\,a_2\,c_1+269\,a_2\,c_2+138\,a_2\,c_3+187\,a_2\,c_4+406\,a_2\,c_5+119\,a_3\,b_1-139\,a_3\,b_2+20\,a_3\,b_3-222\,a_3\,c_1-160\,a_3\,c_2+32\,a_3\,c_3+9\,a_3\,c_4-229\,a_3\,c_5+122\,b_1\,c_1+119\,b_1\,c_2+112\,b_1\,c_3+140\,b_1\,c_4+108\,b_1\,c_5-160\,b_2\,c_1-163\,b_2\,c_2-176\,b_2\,c_3-214\,b_2\,c_4-117\,b_2\,c_5+31\,b_3\,c_1+57\,b_3\,c_2+65\,b_3\,c_3+73\,b_3\,c_4+196\,b_3\,c_5-35\,a_1-21\,a_2+54\,a_3-3\,b_1-3\,b_2+24\,b_3+50\,c_1+46\,c_2+20\,c_3+29\,c_4+63\,c_5-6.$} 

If we take $B=\{1,a_1,a_2,a_3,b_1,b_2\}$ and $B'=\{1,c_1,c_2,c_3$, $c_4,c_5\}$ we obtain
the following known submatrix of $H_{\Lambda}$:
{ $$
\mat{H}^{B',B}_{\Lambda}=
\left[ \begin {array}{cccccc} -6&-35&-21&54&-3&-3
\\ \noalign{\medskip}50&320&292&-222&122&-160\\ \noalign{\medskip}46&
285&269&-160&119&-163\\ \noalign{\medskip}20&134&138&32&112&-176
\\ \noalign{\medskip}29&188&187&9&140&-214\\ \noalign{\medskip}63&382&
406&-229&108&-117\end {array} \right] 
$$}

\noindent{}which is invertible. Thus, the rank is at least $6$. Let us find if 
$H_{\tilde\Lambda}$ can be extended to a rank $6$ Hankel matrix $H_{\Lambda}$.
If we look at $H^{B'^{+},B^{+}}_{\Lambda}$, several coefficients are
unknown. Yet, as will see, they can be determined by exploiting the
commutation relations, as follows.

The columns $\mat{H}^{B',\{m\}}$ are also known for $m\in\{b_3, a_1\,b_1$, $a_2\,b_1,
a_3\,b_1, a_1\,b_2, a_2\,b_2, a_3\,b_2\}$.
Thus we deduce the relations between these monomials and $B$ by solving the system
$$
\mat{H}^{B',B}_{\Lambda}X = \mat{H}^{B',\{m\}}_{\Lambda}.
$$ 
This yields the
following relations in $\Ac_{\Lambda}$:\\
{$b_3 \equiv -1.-0.02486\,a_1+1.412\,a_2+0.8530\,a_3-0.6116\,b_1+0.3713\,b_2, 
a_1\,b_1 \equiv -2.+6.122\,a_1-3.304\,a_2+.6740\,a_3+.7901\,b_1-1.282\,b_2, 
a_2\,b_1 \equiv -2.+4.298\,a_1-1.546\,a_2+1.364\,a_3+.5392\,b_1-1.655\,b_2, 
a_3\,b_1 \equiv -2.-3.337\,a_1+5.143\,a_2+1.786\,a_3-2.291\,b_1+1.699\,b_2, 
a_1\,b_2 \equiv -2.+0.03867\,a_1-0.1967\,a_2+1.451\,a_3-2.049\,b_1+3.756\,b_2, 
a_2\,b_2 \equiv -2.+3.652\,a_1-3.230\,a_2+.9425\,a_3-2.562\,b_1+4.198\,b_2, 
a_3\,b_2 \equiv -2.+6.243\,a_1-7.808\,a_2-1.452\,a_3+5.980\,b_1+0.03646\,b_2$
}

Using the first relation on $b_{3}$, we can reduce $a_1\,b_3, a_2\,b_3, a_3\,b_3$ and obtain
$3$ linear dependency relations between the monomials in
$B\cup\{a_1^{2},a_{1}a_{2},a_{1}a_{3},a_{2}^{2},a_{2}a_{3},a_{3}^{2}\}$.
Using the commutation relations ${\mathrm{lcm}(m_{1},m_{2})\over
  m_{1}}N(m_{1})-{\mathrm{lcm}(m_{1},m_{2})\over m_{2}}N(m_{2})$,
for $(m_{1},m_{2})\in \{ (a_1\, b_1,a_2\, b_1)$, $(a_1\, b_2,a_2\, b_2), (a_2\,b_2, a_3\,b_2)
\}$ where $N(m_{i})$ is the reduction of $m_{i}$ with respect to the prevision relations,
we obtain $3$ new linear dependency relations between the monomials in
$B\cup\{a_1^{2},a_{1}a_{2},a_{1}a_{3},a_{2}^{2},a_{2}a_{3},a_{3}^{2}\}$.
From these $6$ relations, we deduce the expression of the monomials in
$\{a_1^{2},a_{1}a_{2},a_{1}a_{3}$, $a_{2}^{2},a_{2}a_{3},a_{3}^{2}\}$ as
linear combinations of monomials in $B$:\\
{$a_1^2 \equiv 12.08\,a_1-5.107\,a_2+.2232\,a_3-2.161\,b_1-2.038\,b_2-2., 
a_1\,a_2 \equiv 8.972\,a_1-1.431\,a_2+1.392\,a_3-3.680\,b_1-2.254\,b_2-2., 
a_1\,a_3 \equiv -11.56\,a_1+9.209\,a_2+2.802\,a_3+1.737\,b_1+.8155\,b_2-2., 
a_2^2 \equiv -2.+6.691\,a_1+2.173\,a_2+2.793\,a_3-5.811\,b_1-2.846\,b_2, 
a_2\,a_3 \equiv-2.-11.87\,a_1+9.468\,a_2+2.117\,a_3+3.262\,b_1+0.01989\,b_2, 
a_3^2 \equiv -2.+16.96\,a_1-8.603\,a_2+1.349\,a_3-6.351\,b_1-.3558\,b_2.$}

Now, we are able to compute the matrix of multiplication by $a_{1}$ in $B$,
which is obtained by reducing the monomials 
$B\cdot a_{1}=\{a_{1}, a_1^2, a_1\,a_2, a_1\,a_3, a_1\,b_1, a_1\,b_2\}$ by the computed
relations:
{
$$ 
M_{a_{1}}:=\left[ \begin {array}{cccccc}  0.0&- 2.0&- 2.0&- 2.0&- 2.0&- 2.0
\\ \noalign{\medskip} 1.0& 12.08& 8.972&- 11.56& 6.122& 0.03867
\\ \noalign{\medskip} 0.0&- 5.107&- 1.431& 9.209&- 3.304&- 0.1967
\\ \noalign{\medskip} 0.0& 0.2232& 1.392& 2.802& 0.6740& 1.451
\\ \noalign{\medskip} 0.0&- 2.161&- 3.680& 1.737& 0.7901&- 2.049
\\ \noalign{\medskip} 0.0&- 2.038&- 2.254& 0.8155&- 1.282& 3.756
\end {array} \right].
$$}
The eigenvectors of the transposed operator normalized so that the first
coordinate is $1$ are:
{$$ 
\left[ \begin {array}{c}  1.0\\ \noalign{\medskip} 5.0
\\ \noalign{\medskip} 7.003\\ \noalign{\medskip} 3.0
\\ \noalign{\medskip} 3.0\\ \noalign{\medskip}- 4.0\end {array}
 \right] 
,
 \left[ \begin {array}{c}  1.0\\ \noalign{\medskip} 2.999
\\ \noalign{\medskip} 4.0\\ \noalign{\medskip}- 4.999
\\ \noalign{\medskip}- 2.999\\ \noalign{\medskip} 4.999\end {array}
 \right] 
,
 \left[ \begin {array}{c}  1.0\\ \noalign{\medskip} 2.0
\\ \noalign{\medskip} 2.0\\ \noalign{\medskip} 2.0
\\ \noalign{\medskip} 2.0\\ \noalign{\medskip} 2.0\end {array}
 \right] 
,
 \left[ \begin {array}{c}  1.0\\ \noalign{\medskip} 8.001
\\ \noalign{\medskip} 6.002\\ \noalign{\medskip}- 7.002
\\ \noalign{\medskip} 4.001\\ \noalign{\medskip}- 5.001\end {array}
 \right] 
,
 \left[ \begin {array}{c}  1.0\\ \noalign{\medskip}- 1.0
\\ \noalign{\medskip}- 2.0\\ \noalign{\medskip} 3.0
\\ \noalign{\medskip}- 1.0\\ \noalign{\medskip}- 1.0\end {array}
 \right] 
,
 \left[ \begin {array}{c}  1.0\\ \noalign{\medskip} 0.9999
\\ \noalign{\medskip} 0.9999\\ \noalign{\medskip} 0.9999
\\ \noalign{\medskip} 0.9999\\ \noalign{\medskip} 0.9999\end {array}
 \right] 
$$}
They correspond to the vectors of evaluation of the monomial vector $B$ at
the roots of $I_{\Lambda}$. Thus we known the coordinates
$a_{1},a_{2},a_{3},b_{1},b_{2}$ of these roots. By expanding the polynomial\\
{
$\gamma_{1}\,(1+a_1+a_2+a_3))\,(1+b_1+b_2+b_3)\,(1+\cdots)
+\gamma_{2}\,(1-a_1-2\,a_2+3\,a_3)\,(1-b_1-b_2-b_3)\,(1+\cdots)
+\gamma_{3}\,(1+2\,a_1+2\,a_2+2\,a_3)\,(1+2\,b_1+2\,b_2+3\,b_3)\,(1+\cdots)
+\gamma_{4}\,(1+5\,a_1+7\,a_2+3\,a_3)\,(1+3\,b_1-4\,b_2+8\,b_3)\,(1+\cdots)
+\gamma_{5}\,(1+8\,a_1+6\,a_2-7\,a_3)\,(1+4\,b_1-5\,b_2-3\,b_3)\,(1+\cdots)
+\gamma_{6}\,(1+3\,a_1+4\,a_2-5\,a_3)\,(1-3\,b_1+5\,b_2+4\,b_3)\,(1+\cdots)$}\\
(where the $\cdots$ are terms linear in $c_{i}$)
 and identifying the coefficients of $T$ which do not depend on
 $c_{1},\ldots,c_{5}$, we obtain a linear system in $\gamma_{i}$, which unique solution is 
$(2,-1,-2,3$, $-5,-3)$. This allows us to compute the value $\Lambda$ for any
 monomials in $\{a_{1},a_{2},a_{3},b_{1},b_{2},b_{3}\}$. In particular,
 we can compute the entries of $\mat{H}^{B,B}_{\Lambda}$. Solving the system
$\mat{H}^{B,B}_{\Lambda}\, X=\mat{H}^{B,B'}_{\Lambda},$
we deduce the relations between the monomials in $B'$ and $B$ in
$\Ac_{\Lambda}$ and in particular $c_{1},\ldots, c_{5}$ as
linear combinations of monomials in $B$. This allows us to recover the missing coordinates and
yields the following decomposition:\\
$T:=$ {$2\,(1+a_1+a_2+a_3)\,(1+b_1+b_2+b_3)\,(1+c_1+c_2+c_3+c_4+c_5)
-(1-a_1-2\,a_2+3\,a_3)\,(1-b_1-b_2-b_3)\,(1-c_1-2\,c_2-3\,c_3-4\,c_4+5\,c_5)
-2\,(1+2\,a_1+2\,a_2+2\,a_3)\,(1+2\,b_1+2\,b_2+3\,b_3)\,(1+2\,c_1+2\,c_2+2\,c_3+2\,c_4+2\,c_5)
+3\,(1+5\,a_1+7\,a_2+3\,a_3)\,(1+3\,b_1-4\,b_2+8\,b_3)\,(1+4\,c_1+5\,c_2+6\,c_3+7\,c_4+8\,c_5)
-5\,(1+8\,a_1+6\,a_2-7\,a_3)\,(1+4\,b_1-5\,b_2-3\,b_3)\,(1-6\,c_1-5\,c_2-2\,c_3-3\,c_4-5\,c_5)
-3\,(1+3\,a_1+4\,a_2-5\,a_3)\,(1-3\,b_1+5\,b_2+4\,b_3)\,(1-3\,c_1-2\,c_2+3\,c_3+3\,c_4-7\,c_5).$}


\end{document}